\documentclass[a4paper,11pt]{article} 

\usepackage{amsmath}    
\usepackage{amssymb}
\usepackage{amsthm} 
\usepackage{bbm}
\usepackage{booktabs}
\usepackage{microtype}
\usepackage{a4wide}
\usepackage{todonotes}

\newtheorem{theorem}{Theorem}[section]

\newtheorem{lemma}[theorem]{Lemma}
\newtheorem{proposition}[theorem]{Proposition}

\newtheorem{assumption}[theorem]{Assumption}
\theoremstyle{definition}
\newtheorem{definition}[theorem]{Definition}
\newtheorem{remark}[theorem]{Remark}

\usepackage[toc,page]{appendix}
\usepackage{color}
\usepackage{dsfont}
\usepackage{extarrows}
\usepackage{enumerate}
\usepackage{graphicx}   
\usepackage{mathrsfs}
\usepackage{verbatim}   
\usepackage{subfigure}  
\usepackage{hyperref}   

\newcommand{\Tr}{\mathrm{Tr}}

\def \C{\mathbb{C}}

\def \E{\mathbb{E}}
\def \F{\mathbb{F}}

\def \H{\mathbb{H}}

\def \N{\mathbb{N}}

\def \P{\mathbb{P}}

\def \R{\mathbb{R}}

\def\Bc{{\cal B}}

\def\Fc{{\cal F}}
\def\Gc{{\cal G}}
\def\Hc{{\cal H}}
\def\Ic{{\cal I}}
\def\Jc{{\cal J}}

\def\Sc{{\cal S}}

\def \d {\delta}

\def \id {\mathds{1}}
\def \i {{\rm i}}

\def \d {{\rm d}}
\def \dist {\mathsf{dist}}

\title{\Large \bf
Stabilization of jump-diffusion stochastic differential equations by hysteresis switching
}
\date{\today}

\author{Weichao Liang
\thanks{\small Xi'an Jiaotong University, Faculty of Electronic and Information Engineering, School of Automation Science and Technology \texttt{(weichao.liang@xjtu.edu.cn).}}%
\and 
Gaoyue Guo
\thanks{
{\small Université Paris-Saclay CentraleSupélec, Laboratoire MICS and CNRS FR-3487 \texttt{(gaoyue.guo@centralesupelec.fr)}. Guo acknowledges financial support  \emph{Bourse ``Systemic Robustness and Systemic Failure''} by \href{https://www.institutlouisbachelier.org/adn-de-lilb/les-fondations/}{\emph{Institut Europlace de Finance}}.}}
}

\begin{document}


\maketitle

\begin{abstract}
We address the stabilization of both classical and quantum systems modeled by jump-diffusion stochastic differential equations using a novel hysteresis switching strategy. Unlike traditional methods that depend on global Lyapunov functions or require each subsystem to stabilize the target state individually, our approach employs local Lyapunov-like conditions and state-dependent switching to achieve global asymptotic or exponential stability with finitely many switches almost surely. We rigorously establish the well-posedness of the resulting switched systems and derive sufficient conditions for stability. The framework is further extended to quantum feedback control systems governed by stochastic master equations with both diffusive and jump dynamics. Notably, our method relaxes restrictive invariance assumptions often necessary in prior work, enhancing practical applicability in experimental quantum settings. Additionally, the proposed strategy offers promising avenues for robust control under model uncertainties and perturbations, paving the way for future developments in both classical and quantum control.
\end{abstract}
\section{Introduction}\label{sec:intro}
Stochastic differential equations (SDEs) driven by both Wiener and Poisson processes, commonly referred to as jump-diffusion SDEs, provide a versatile framework for modeling systems influenced simultaneously by continuous fluctuations and discrete random perturbations~\cite{applebaum2009levy,protter2004stochastic}.  Stabilization of such systems is inherently challenging due to the complex interaction between diffusive noise and jump dynamics~\cite{applebaum2009asymptotic}. 

Switching control, which enables transitions between multiple controllers based on the system's state, provides enhanced flexibility in handling nonlinear dynamics and large uncertainties~\cite{teel2014stability,sun2011stability,liberzon2003switching}. In particular, hysteresis switching~\cite{morse1992applications} effectively mitigates undesirable chattering, thereby improving robustness to noise and measurement errors and offering stability guarantees. These benefits are especially relevant in systems influenced simultaneously by Brownian motion and Poisson jumps, where the simultaneous presence of continuous and discrete stochastic elements introduces additional complexity.

In this work, we investigate the stabilization of jump-diffusion SDEs via state-dependent hysteresis switching laws. Unlike time-scheduled or Markovian switching, hysteresis-based strategies rely on the evolution to determine switching, reducing sensitivity to noise-induced switching and improving practical applicability.
Existing frameworks for the stability analysis of nonlinear switched systems driven by diffusion processes have been developed using common Lyapunov functions or multiple Lyapunov functions combined with comparison principles~\cite{chatterjee2006stability,wu2013stability,ren2021stability}. However, there has been limited study on the stabilization of switched jump-diffusion systems. Moreover, constructing common or multiple Lyapunov functions can be technically difficult, especially when the target state is not invariant under each individual subsystem. To address this challenge, we propose a new approach that does not require each subsystem to stabilize the target state individually. Instead, we assume that for each state, there exists at least one controller capable of steering the system toward the target state. Combining these local behaviors using stochastic trajectory analysis, we demonstrate that global asymptotic stability (GAS) or even global exponential stability (GES) can be achieved based on local Lyapunov-like arguments.

As a concrete application, we apply our theoretical framework to a class of jump-diffusion stochastic master equations (SMEs) arising in quantum feedback control under continuous measurements~\cite{breuer2002theory,wiseman2009quantum,belavkin1989nondemolition}. A central challenge in quantum control lies in stabilizing the system's quantum state—governed by SMEs—toward a desired pure state or subspace~\cite{altafini2012modeling}. This task becomes significantly more complex when the innovation processes include discontinuities induced by quantum jumps.

To address quantum stabilization, a variety of feedback strategies have been developed. These include Markovian feedback (output-feedback) methods~\cite{ticozzi2008quantum,ticozzi2009analysis,wiseman2009quantum}, and Bayesian feedback (filtering-based) methods~\cite{van2005feedback,mirrahimi2007stabilizing,ticozzi2012stabilization,liang2019exponential,liang2024model,liang2025exploring}.  For Markovian feedback, necessary and sufficient conditions for global asymptotic stability (GAS) have been established~\cite{altafini2012modeling}. However, these approaches are often difficult to implement in practice due to experimental complexity. Similarly, Bayesian feedback strategies typically rely on quantum non-demolition (QND) measurements, which remain technically demanding in real-world settings. Moreover, most existing feedback designs rely on Hamiltonian modulation, while dissipative mechanisms have traditionally been treated either as passive background dynamics or open-loop stabilization tools.
However, recent advances have demonstrated that engineered dissipation—when carefully controlled—can play an active and powerful role in quantum information tasks such as state preparation and quantum error correction~\cite{dissipativeQC,dissipativeencoding}. Motivated by these developments, recent work has explored dissipative switching control for diffusion-type SMEs~\cite{scaramuzza2015switching,grigoletto2021stabilization,liang2024switching}, combining coherent and dissipative resources. Dissipation is introduced via controlled interactions with engineered environments and measurement processes. While earlier results often treated switching as a technical tool to eliminate undesired equilibria~\cite{mirrahimi2007stabilizing,ticozzi2012stabilization}, our approach takes a fundamentally different view: we consider switching as an active control resource, explicitly designed and analyzed to robustly stabilize quantum systems in the presence of both continuous and jump stochastic dynamics. 
Nevertheless, existing dissipative switching control schemes for diffusion SMEs typically assume that the target state or subspace is invariant and globally asymptotically stable under each individual subsystem, a condition closely related to the multiple Lyapunov function framework in classical switched SDEs~\cite{teel2014stability}. In practice, however, this assumption is often difficult to verify or satisfy, particularly in quantum systems with complex dynamics and measurement imperfections.

Building on our framework for switching control of jump-diffusion SDEs, we extend it to the setting of open quantum systems undergoing continuous measurement, modeled by jump-diffusion SMEs. Our approach achieves stabilization without requiring the target state to be invariant or GAS for each subsystem. Instead, by analyzing the stochastic trajectories and using local Lyapunov-like arguments, we design switching strategies that ensure global asymptotic or exponential stability. This leads to a more practical and experimentally implementable solution for quantum feedback control.

Our main contribution is summarized as follows.
\begin{enumerate}
    \item \textbf{Well-posedeness}: We show in Theorem~\ref{thm:wellposed} that there exists a unique (strong) solution to switched jump-diffusion SDEs and SMEs.
    \item \textbf{Global Asymptotic/Exponential Stabilization}: We derive sufficient conditions and design switching laws $\sigma_1$ (for jump-diffusion SDEs) and $\sigma_2$ (for jump-diffusion SMEs) that ensure the GAS of the target state, summarized in Theorem~\ref{Theorem:GAS} and Proposition~\ref{Prop:GAS_SME}. Furthermore, by imposing an additional condition on the local behavior around the target state for one subsystem, we establish GES in Propositions~\ref{Proposition:GES} and~\ref{Proposition:GES_SME}.
\end{enumerate}
In addition to the rigorous proofs, we also provide intuitive insights explanations to clarify the rationale behind the proposed conditions in ensuring stabilization.

\subsection{Preliminaries: Switched stochastic differential equations}\label{ssec:prelimi}

This paper investigates the stabilization of systems governed by jump-diffusion stochastic differential equations (SDEs) of the form:
\begin{equation}\label{Eq:SDE_X}
    \d X_t=f\big(X_{t-}\big)\d t+g_k\big(X_{t-}\big)\d W_t+\int_{|x|\leq c}h\big(X_{t-},x\big)\tilde{N}(\d t,\d x),
\end{equation}
where $f, g:\R^{d\times d}\to \R^{d\times d}$ and $ h:\R^{d\times d}\times\mathbb{R}\longrightarrow \R^{d\times d}$ are measurable functions, $W$ denotes a real-valued  Brownian motions, $N$ stands for an independent Poisson random measure on $\mathbb{R}_+\times\R$ of intensity  $\d t \d x$ and $\tilde N$ is its compensated measure, i.e. $\tilde N(\d t,\d x):=N(\d t,\d x)-\d t \d x$. Well-posedness of~\eqref{Eq:SDE_X} under Lipschitz conditions can be ensured by standard arguments, see e.g.~\cite{applebaum2009levy}.

In many practical situations, constructing a global Lyapunov function to ensure stability is nontrivial. To address this challenge, we propose an alternative approach based on hysteresis switching strategies. This leads to  (controlled) SDEs defined on some  filtered probability space $(\Omega,\mathcal{F},\F=(\mathcal{F}_t)_{t\ge 0},\mathbb{P})$ which satisfies the usual hypotheses of completeness and right continuity,
\begin{equation}\label{Eq:SDE}
    \d X_t=\sum^m_{k=1}u^k_t\left( f_{k}\big(X_{t-}\big)\d t+g_k\big(X_{t-}\big)\d W^k_t+\int_{|x|\leq c}h_k\big(X_{t-},x\big)\tilde{N}^k(\d t,\d x)\right),
\end{equation}
where $W^1,\ldots, W^m$ are independent real-valued Brownian motions, $N^1,\ldots, N^m$ are independent Poisson random measures on $\mathbb{R}_+\times\R$, also independent of $W^1,\ldots, W^m$, with the same intensity $\d t\d x$, and $\tilde{N}^k(\d t,\d x):=N^k(\d t,\d x)-\d t \d x$. 

The control processes $u^k_t\in \{0,1\}$ represent bang-bang switching rules satisfying
$$\sum_{k=1}^m u^k_t =1,\quad \forall t\ge 0,$$
i.e., exactly one subsystem is active at each time, and their definition will be specified later.

Formally, we consider $m$ open subsets $\Theta_1,\ldots \Theta_m\subset \R^{d\times d}$ that form a partition of the state space $\R^{d\times d}$.
A switching event is triggered when the trajectory $X_t$ crosses a predefined switching surface. The objective is to design a switching law $u_t^k$ and analyze the stability of the solution $X_t$ with respect to a target state.


Inspired by the scale-independent hysteresis switching logic~\cite[Chapter 6]{liberzon2003switching}, we consider the regions $\Theta_1,\ldots, \Theta_m$, where the $p$-th subsystem is active when the system state lies in region $\Theta_p$, that is,
$$u_t^{k} = \id_{\{k=p\}},\quad \mbox{if } X_t\in \Theta_p,$$
where $\id$ denotes the indicator function. The region $\Theta_p$ is referred to as the \textit{active region} of $X_t$. 

However, since the sets $\Theta_1,\ldots, \Theta_m$ are open, it is possible for $X_t$ to lie on the boundary of multiple regions, potentially resulting in two distinct indices $k\neq k'$ such that $u^k_t=u^{k'}_t=1$, which contradicts the switching rule requiring exactly one active subsystem at each time. To address  this issue, we construct more refined switching rules to different scenarios where $\{X_t \in \Theta_k\}$ for $k\in [m]:=\{1,\ldots, m\}$, leading to corresponding coefficients  $f_k, g_k, h_k$. 

Before proceeding to control design and stability analysis, we first establish the well-posedness of the switched system~\eqref{Eq:SDE} under suitable structural assumptions. In particular, we assume that the open sets $\Theta_1,\ldots, \Theta_m$ satisfies the following conditions.

\begin{assumption}[Partition with no‐crossing jumps]\label{ass:partition} $\Theta_1,\dots,\Theta_m\subset\R^{d\times d}$ form a partition, i.e.,
\[
\bigcup_{k=1}^m \Theta_k\;=\;\R^{d\times d},
\]
Writing $P(X)=\{k:X\in \Theta_k\}$ for every $X\in\R^{d\times d}$, there exists  $k\in P(X)$ such that
\begin{enumerate}[(i)]
  \item \emph{(Exact invariance under jumps)}
  \[
    X + h_k\bigl(X,x\bigr)\in\Theta_k, \quad\forall\,|x|\le c.
  \]
  Equivalently, 
      \[
           X_{t-}\in\Theta_k\Longrightarrow X_{t-}+\int_{|x|\le c}h_k(X_{t-},x)\,N(\{t\},dx)\in\Theta_k.
      \]
  \item \emph{(Practical sufficient verification)}  
  A convenient (but more restrictive) way to guarantee~(i) is to impose the metric bound
  \[
    \dist(X,\partial \Theta_k):=\inf \big\{|X-Y|: Y\in \partial \Theta_k\big\}> \sup\big\{\|h_k(X,x)\|: |x|\leq c\big\}.
  \]
\end{enumerate}
\end{assumption}

\begin{remark}
  Assumption~\ref{ass:partition} allows us to avoid Zeno and chattering phenomenon induced by the continuous part of $X_t$. If $X_t$ were purely continuous, any two switching times would be
  separated by a nonzero time interval, hence no Zeno or chattering phenomenon.  However, in system~\eqref{Eq:SDE} a jump and a switch can in principle occur simultaneously,  potentially causing two switches at the same instant.  Inspired by~\cite{krystul2006stochastic}, we impose Assumption~\ref{ass:partition} so that all switching surfaces can only  be hit via the continuous evolution of~$X_t$.  Note that without Assumption~\ref{ass:partition} the solution remains well–defined, although Zeno or chattering phenomenon may occur when jumps cross region boundaries.

  Furthermore, in practice one may choose the exact invariance condition (i) for some regions and the simpler metric test (ii) for others, depending on which is easier to verify in each case.  In
  Assumption~\ref{ass:partition2} and ~\ref{ass:partition3} below we give concrete examples in which one $\Theta_k$ are treated by (ii) and the remaining by (i).
\end{remark}

Before stating our well-posedness result, we first define the switching law $u_t=(u^1_t,\ldots, u^m_t)$. Let $D(\R_+,\R^{d\times d})$ be the space of Skorokhod space endowed with the Skorokhod topology. A functional $\Phi:\R_+\times D(\R_+,\R^{d\times d})\to \R^m$ is said to be predictable process if $\Phi$ is measurable and satisfies for all $t\ge 0$ and $\omega,\omega'\in D(\R_+,\R^{d\times d})$
$$\Phi(t,\omega)=\Phi(t,\omega') \mbox{ whenver } \omega(s)=\omega'(s) \mbox{ for } s\in [0,t).$$
Provided the activation regions $\Theta_1,\ldots, \Theta_m$, we construct the following functional $\Phi=(\phi^1,\ldots, \phi^m)$ as follows: $\phi^k(t,\omega)\in  \{0,1\}$  and are defined recursively. Set $T_0(\omega):=0$ and 
\begin{equation*}
   k_0(\omega):= \min\big\{k\in [m]:~ \omega(0)\in \Theta_k \big\}.
\end{equation*}
Note that $k_0(\omega)<\infty$ holds under Assumption \ref{ass:partition}. Then, define $\phi^k(t,\omega):=\mathds{1}_{\{k=k_0(\omega)\}}$ for $t\le T_1(\omega)$, where $T_1(\omega):=\inf\{t>T_0(\omega): \omega(t)\notin \Theta_{k_0(\omega)}\}$
and 
\begin{equation*}
   k_1(\omega):= \min\big\{k\in [m]:~ \omega(T_1(\omega))\in \Theta_k \big\}, 
\end{equation*}
So $\phi^k(t,\omega):=\mathds{1}_{\{k=k_1(\omega)\}}$ for $T_1(\omega)<t\le T_2(\omega)$, where $T_n(\omega)$, and $k_n(\omega)$ can be defined similarly for $n\ge 2$. 
\begin{remark}
It is worth noting that, for every $\omega\in D(\R_+,\R^{d\times d})$ such that $\Delta \omega(t):=\omega(t)-\omega(t-)\le c$ for all $t\ge 0$, one must have either $T_{n^*}=\infty$ for some $n^*\in \N$ or $T_{n-1}< T_n$ for all $n\in\N$.
\end{remark}
Finally, we define the switching laws by setting $u_t=\Phi(t,X)$ for $t\in\R_+$ and $k\in [m]$. 

\subsection{Preliminaries: Switched stochastic master equations}\label{ssec:prelimi_sme}

A quantum trajectory is a solution of a matrix-valued stochastic differential equation which describes the time evolution of a quantum system undergoing continuous measurements. These equations are known as stochastic master equations (SMEs) or Belavkin's equations. In the literature,  two main types of SMEs are commonly studied: the diffusive type and the jump (Poisson-type) SME.
To state our results, we first outline the mathematical framework underlying the theory of \emph{quantum filtering}. Section~\ref{ssec:prelimi} provides the necessary background on switched stochastic differential equations, which parallels and supports the current setting.
We consider open quantum systems defined on a Hilbert space $\H$ of dimension $d\in \N$. Let $\Bc(\H)$\footnote{For technical reasons we are restricted to finite-dimensional Hilbert spaces. Note that both $\H$ and $\Bc(\H)$ can be identified as the standard Euclidean spaces of suitable dimension. As a result, we can invoke known results in Euclidean space without elaboration.} be the space of (linear) bounded operators on $\H$ (where we note that $\Bc(\H)$ is isomorphic to $\R^{d\times d}$ and all its norms are equivalent), and we define respectively
\begin{align*}
    \Bc_*(\H)&:=\{A\in\Bc(\H):~ A=A^* \},\\
    \Bc_+(\H)&:=\{A\in\Bc_*(\H):~A\geq 0\},\\
    \Sc(\H)&:=\{A\in\Bc_+(\H):~\mathrm{Tr}[A]=1\},
\end{align*}
where $A^*$ denotes the adjoint operator of $A$. All the elements of $\Sc(\H)$ are called \emph{density matrices},  and are used to describe the state of the quantum system under consideration. The commutator and anti-commutator are defined as follows: for $A, B\in \Bc(\H)$,  
$$[A,B]:=AB-BA \quad \mbox{and} \quad \{A,B\}:=AB + BA.$$ 
For $C\in\Bc(\H)$, we define
 the maps $\Ic_C, \Jc_C: \Bc(\H)\to \Bc(\H)$ and $v_C:\Bc(\H)\to\C$  that will be repeatedly used throughout the paper: for $A\in \Bc(\H)$,
$$\Ic_C(A):=CA C^*-\frac12 \{ C^*C, A\}, \quad \mathcal{J}_C(A)  :=\frac{CA C^*}{v_C(A)} \quad \mbox{and}\quad v_C(A):=\mathrm{Tr}[CA C^*].$$ 
In particular, $\Ic_C, v_C$ are linear maps and $\Jc_C : \Bc_+(\H) \hookrightarrow\Sc(\H)$, $v_C : \Bc_+(\H) \hookrightarrow\R_+$.
Then, we define the following maps:
\begin{align*}
    \mathcal{F}(A)&:=-\i[H,A]+\Ic_{L}(A)+\Ic_{C}(A)+\Ic_{D}(A), \\   
    \mathcal{G}_C(A)&:=CA+AC^*-\mathrm{Tr}[(C+C^*)A]A,\\
    \mathcal{H}_D(A)&:=  \mathcal{J}_{D}(A)-A.
\end{align*}
where $\i:=\sqrt{-1}$.

A quantum system is characterized by its Hamiltonian $H\in \Bc_*(\H)$ and noise operators $L, C, D\in\Bc(\H)$. The corresponding SME is given by 
\begin{align}
    \d\rho_t=\mathcal{F}(\rho_{t-}) \d t  +\mathcal{G}_C(\rho_{t-})\d W_t + \int_{\mathbb{R}}\mathds{1}_{\{0\le x\le v_{D}(\rho_{t-}) \}} \mathcal{H}_D(\rho_{t-})\tilde N(\d t,\d x), \label{eq:sme0}
\end{align}

Motivated by both theoretical and practical significance, well-posedness of \eqref{eq:sme0} has emerged as a fundamental question and has drawn abundant attention. Using the fact that the solution $\rho_t\in \Sc(\H)$ yields $v_{D}(\rho_{t-})\le c$ for some $c>0$ large enough and thus
$$\int_{\mathbb{R}}\mathds{1}_{\{0\le x\le v_{D}(\rho_{t-}) \}} \mathcal{H}_{D}(\rho_{t-})\tilde N(\d t,\d x)=\int_{|x|\le c}\mathds{1}_{\{0\le x\le v_{D}(\rho_{t-}) \}} \mathcal{H}_{D}(\rho_{t-})\tilde N(\d t,\d x),$$
\eqref{eq:sme0} formally appears as a specific case of general jump-diffusion SDEs. However, the following observations preclude the direct application of the general well-posedness theory:
\begin{itemize}
    \item $A\mapsto \mathrm{Tr}[(C+C^*)A]A$ is not (globally) Lipschitz;
    \item $A\mapsto \Jc_D(A)$ is not Lipschitz and $(A,x)\mapsto \mathds{1}_{\{0\le x\le v_{D}(A) \}}$ is not continuous.
\end{itemize}
Fortunately, well-posedness of \eqref{eq:sme0} has been established by Pellegrini in \cite{pellegrini2010markov} which ensures that there exists a unique \emph{strong solution} $\rho=(\rho_t)_{t\ge 0}$ to \eqref{eq:sme0} so that $\rho_t\in\Sc(\H)$ if $\rho_0\in \Sc(\H)$. 

In realistic experimental settings, constructing a global Lyapunov function to ensure stabilization of a quantum system toward a target state or subspace under jump-diffusion stochastic master equations (SMEs) is often infeasible due to mathematical and physical complexity. To address this challenge, we adopt the hysteresis switching strategies proposed in Section~\ref{ssec:prelimi}. We consider a quantum system characterized by a free Hamiltonian $H_0\in \Bc_*(\H)$, coupled to a control field via a control Hamiltonian $H_k\in \Bc_*(\H)$, and interacting with an external reservoir described by the noise operator $L_k\in\Bc(\H)$. The system is continuously monitored through indirect measurements, including homodyne/heterodyne detection and photon counting, modeled respectively by noise operators $C_k, D_k\in\Bc(\H)$. The evolution of the system under continuous measurement and switching control is described by the following SME:
\begin{equation}\label{eq:sme_switch}
    \d \rho_t=\sum^m_{k=1}u^k_t\left( \Fc_k(\rho_{t-})\d t +\mathcal{G}_k(\rho_{t-})\d W^k_t+ \int_{\mathbb{R}}\mathcal{H}_k(\rho_{t-})\mathds{1}_{\{0\le x\le v_{D_k}(\rho_{t-}) \}}\tilde N^k(\d t,\d x)\right),
\end{equation}
where $u^k_t\in \{0,1\}$ represents bang-bang switching rules related to the $m\in\mathbb{N}$ open subsets $\Theta_1,\ldots \Theta_m\subset \mathcal{S}(\mathbb{H})$. Under suitable switching rules, each region $\Theta_k$ corresponds to coefficients 
\begin{align*}
    \mathcal{F}_k(A)&:=-\i[H_0+H_k,A]+\Ic_{L_k}(A)+\Ic_{C_k}(A)+\Ic_{D_k}(A), \\   
    \mathcal{G}_k(A)&:=C_kA+AC_k^*-\mathrm{Tr}[(C_k+C_k^*)A]A,\\
    \mathcal{H}_k(A)&:=  \mathcal{J}_{D_k}(A)-A.
\end{align*}

\section{Well-posedness}\label{sec:well}
Our first result concerns the well-posedness of the switched general jump-diffusion SDE~\eqref{Eq:SDE} and switched jump-diffusion SME \eqref{eq:sme_switch}.  We adopt the following definition. 
\begin{definition}
A process $X$ on $(\Omega,\mathcal{F},\F=(\mathcal{F}_t)_{t\ge 0},\mathbb{P})$ is said to be a  strong solution to 
\begin{equation}
    X_t= X_0+ \int_0^t f\big(X_{s-}\big)\d s+ \int_0^t g\big(X_{s-}\big)\d W_s+\int_0^t \int_{|z|\leq c}h\big(X_{s-},z\big)\tilde{N}(\d s,\d z),\quad \forall t\ge 0.
\end{equation}
if $\P\big(X\in D(\R_+,\R^{d\times d}) \big)=1$ and $X$ is adapted. 
\end{definition}
Then one has the following well-posedness result.
\begin{theorem}\label{thm:wellposed}
Under Assumption~\ref{ass:partition}, there exists a unique strong solution to \eqref{Eq:SDE} if one of the following conditions holds:
\begin{itemize}
    \item[\emph{(A)}] $X_0\in\Sc(\H)$ and $f_k=\Fc_k, g_k=\Gc_k, h_k=\Hc_k$ for $k\in [m]$;
    \item[\emph{(B)}] $\E[|X_0|^2]<\infty$ and there exists $L>0$ so that
    $$\sum_{k=1}^m \left( |f_k(x)-f_k(y)|^2 + |g_k(x)-g_k(y)|^2 + \int_{|z|\leq c}|h_k(x,z)-h_k(y,z)|^2\d z\right) \le L|x-y|^2,$$
for all $x,y\in\R^{d\times d}$.
\end{itemize}
In particular, there exists a unique solution $\rho$ to \eqref{eq:sme0} with $\rho_t\in \Sc(\H)$ for all $t\ge 0$.
\end{theorem}
\begin{proof}
The proof relies on the well-posedness result for jump-diffusion SDEs with Lipschitz coefficients under (A) (resp. for the SME \eqref{eq:sme0} considered by Pellegrini \cite{pellegrini2010markov} under (B)). For the sake of simplicity we only consider the case under (A), and start by showing the existence.  Consider the SDE 
$$Y_t = Y_0 + \int_0^t f_{k_{0}}\big(Y_{s-}\big)\d s +\int_0^t  g_{k_{0}}\big(Y_{s-}\big)\d W^{{k_{0}}}_s+ \int_0^t \int_{|x|\leq c}h_{{k_{0}}}\big(Y_{s-},x\big)\tilde{N}^{{k_{0}}}(\d s,\d x),\quad \forall t>0,$$
where $k_0$ (that might be random) is chosen according to
$$k_0:=\min\big\{k\in [m]:~ X_0\in \Theta_k\big\}.$$
This is a standard SDE and the Lipschitz conditions on $f_{k_{0}},g_{k_{0}}, h_{k_{0}}$ ensure the unique solution (for the second case under (B) it suffices to apply the existence and the uniqueness established in \cite{pellegrini2010markov,barchielli2009quantum}), denoted by $Y^0$. Define then
\begin{align*}
&\tau_1:=\inf\big\{t\ge 0:~ Y^{0}_t\notin \Theta_{k_{0}}\big\}, \quad k_1:=\min\Big\{k\in [m]:~ Y^{0}_{\tau_{1}}\in \Theta_k \Big\}.
\end{align*}
For $n\ge 2$, define respectively 
\begin{align*}
&\tau_n:=\inf\big\{t\ge \tau_{n-1}:~ Y^{n-1}_t\notin \Theta_{k_{n-1}}\big\}, \quad k_n:=\min\Big\{k\in [m]:~ Y^{n-1}_{\tau_n}\in \Theta_k \Big\},   
\end{align*}
where $Y^{n-1}$ is the unique solution to 
$$\d Y_t = f_{k_{n-1}}\big(Y_{t-}\big)\d t + g_{k_{n-1}}\big(Y_{t-}\big)\d W^{{k_{n-1}}}_t+ \int_{|x|\leq c}h_{{k_{n-1}}}\big(Y_{t-},x\big)\tilde{N}^{{k_{n-1}}}(\d t,\d x),\quad \forall t>\tau_{n-1},$$
such that $Y_{\tau_{n-1}}=Y^{n-2}_{\tau_{n-1}}$. In particular, Assumption \ref{ass:partition} gives that $\tau_n<\infty$ and $\tau_{n-1}<\tau_n$ for all $n\in\N$.
We define finally the process $X$ by
\begin{equation}\label{def:sol}
 X_t :=\sum_{n\ge 1}\id_{[\tau_{n-1}, \tau_n)}(t) Y^{n-1}_t,\quad \forall t\ge 0. 
\end{equation} 
By a straightforward verification, one may verify $X$ solves \eqref{Eq:SDE} and further   
\begin{equation}
\{0,1\}\ni u^k_t:=\sum_{n\ge 1}\id_{(\tau_{n-1}, \tau_n]}(t)\id_{\{k=k_{n-1}\}},\quad \forall t\ge 0. \label{def:switch}  
\end{equation}
Next we turn to prove the uniqueness. If $X'$ stands for a solution to \eqref{Eq:SDE}, define the sequence of stopping times $(\tau_n')_{n\ge 0}$ with $\tau_0':=0$, $k_0':=k_0$ and
\begin{align*}
&\tau_n':=\inf\big\{t\ge \tau_{n-1}':~ X_t\notin \Theta_{k_{n-1}}\big\}, \quad k_n':=\min\Big\{k\in [m]:~ X'_{\tau_n'}\in \Theta_k \Big\}.  
\end{align*}
With the help of this localization argument, one deduces immediately $X_t=X'_t$ for $t< \tau_1\wedge \tau_1'$, using the classical arguments under (A) and using the reasoning in the proof of Theorem~4 by \cite{pellegrini2010markov} under (B), and further $\tau_1=\tau_1'$ using the path uniqueness over $[0,\tau_1\wedge \tau_1')$ for the diffusive case estbalished in \cite{barchielli2009quantum}, and finally $X_{\tau_1}=X'_{\tau_1}$ by the observation that the jumps  $\Delta X_{\tau_1}, \Delta X'_{\tau_1}$ are the same function of $X_{\tau_1-}=X'_{\tau_1-}$. Repeating the arguments on the intervals $[\tau_{n-1},\tau_n), [\tau'_{n-1},\tau_n')$ using recursively $\tau_{n-1}=\tau'_{n-1}$, we deduce that $\tau_n=\tau_n'$ and $X|_{[\tau_{n-1},\tau_{n}]}=X'|_{[\tau_{n-1},\tau_{n}]}$. The proof is  fulfilled by pasting the solution over all the intervals. 
\end{proof}
\begin{remark}
In contrast to the findings presented in~\cite{wu2013stability,ren2021stability}, which focus on diffusion-type switched stochastic systems, our work investigates state-dependent switching for systems with both diffusion and jump dynamics. A key novelty lies in the flexibility of our switching laws. While the works of \cite{wu2013stability,ren2021stability} exclusively activate a single subsystem at each switching event across all sample paths, our paper adopts a more general approach. Specifically, during a switching even, a subsystem is individually designated for almost every sample path. This distinction allows for a more tailored enhancement of system performance through the application of the hysteresis switching.
\end{remark}

\section{Stabilization by hysteresis switching}\label{Sec:Hysteresis}

In order to make full use of diffusion and jump terms of the stochastic system for enhanced convergence and reduced switching frequency, inspired by~\cite[Section III.B]{liang2024switching} and~\cite[Theorem 2]{grigoletto2021stabilization}, we propose the state-dependent switching law ensuring GAS of the target state $\bar{x}$, where the switching stops in finite time almost surely. 
We consider the following three types of stability~\cite{khasminskii2011stochastic,mao2007stochastic}.
\begin{definition}\label{Def:Stability} 
The state $\bar{x}\in\mathbb{R}^{d\times d}$ is said to be 
\begin{enumerate}
    \item
\emph{stable in probability}, if for every pair $\varepsilon \in (0,1)$ and $r>0$, there exists $\delta = \delta(\varepsilon,r,t_0)>0$ such that,
\begin{equation*}
\mathbb{P} \big( |X_t-\bar{x}|<r \text{ for } t \geq 0 \big) \geq 1-\varepsilon,
\end{equation*}
whenever $|x_0-\bar{x}|<\delta$. 

\item almost surely \emph{globally asymptotically stable} (GAS), if it is stable in probability and,
\begin{equation*}
\mathbb{P} \left( \lim_{t\rightarrow\infty}|X_t-\bar{x}|=0 \Big|X_0=x\right) = 1, \quad \forall x\in\mathbb{R}^{d\times d}.
\end{equation*}

\item almost surely
\emph{globally exponentially stable} (GES), if
\begin{equation*}
\P\left(\limsup_{t \rightarrow \infty} \frac{1}{t} \log |X_t-\bar{x}| < 0\Big|X_0=x\right)=1,\quad \forall x\in\mathbb{R}^{d\times d}.
\end{equation*}
The left-hand side of the above inequality is called the \emph{sample Lyapunov exponent}.
\end{enumerate}
\end{definition}

\subsection{Design of switching rules and stability analysis}  
\label{sec:switchingrules_SDE}

Define $\Lambda_l:=\{x\in\mathbb{R}^{d\times d}: |x|<l\}$ with $l>0$ and $\bar{\Lambda}_l$ the closure of $\Lambda_l$. Denote by $\mathcal{A}_k$ the infinitesimal generator of the $k$-subsystem
$$\d X_t= f_{k}(X_{t-})\d t+g_k(X_{t-})\d W^k_t+\int_{|x|\leq c}h_k(X_{t-},x)\tilde{N}^k(\d t,\d x).$$  
We make the following control hypothesis ensuring the hysteresis property.
\begin{assumption}
Let $\mathcal K$ be the family of all continuous non-decreasing functions $\mu:\mathbb{R}_{+}\rightarrow\mathbb{R}_{+}$ such that $\mu(0)=0$ and $\mu(r)>0$ for all $r>0$. 
 \begin{itemize}
\item[\textbf{\emph{H1}}:] There exist  functions $\mu_1,\mu_2,\nu\in\mathcal{K}$, and $\mathsf{V}\in\mathcal{C}^2(\mathbb{R}^{d\times d},\mathbb{R}_+)$ such that $\mathsf{V}(x)=0$ if and only if $x=\bar{x}$,  and a constant $l>0$ and $\mathbf{j}\in[m]$ such that $\mu_1(|x|)\leq \mathsf{V}(x)\leq \mu_2(|x|)$ and $\mathcal{A}_{\mathbf{j}}\mathsf{V}(x)\leq -\nu(|x|)$ for all $x\in\bar{\Lambda}_l$.
\item[\textbf{\emph{H2}}:] There exist a constant $\delta>0$ and $V\in\mathcal{C}^2(\mathbb{R}^{d\times d},\mathbb{R}_+)$ such that $V(x)=0$ if and only if $x=\bar{x}$, and $\mathbf{A}_V(x):=\min_{k\in [m]}\mathcal{A}_k V(x)<-\delta$ for all $x\notin \Lambda_{l^*-\epsilon}$, where  $l^*:=\mu_2^{-1}\circ\mu_1(l)\in(0,l]$ and $\epsilon\in(0,l^*)$.
\end{itemize}   
\end{assumption} 
If \textbf{H1} is satisfied, it is straightforward to deduce that $\bar{x}$ is a trivial solution for $\mathbf{j}$-th subsystem, i.e., $f_{\mathbf{j}}(\bar{x})=g_{\mathbf{j}}(\bar{x})=\int_{|z|\leq c}h_{\mathbf{j}}(\bar{x},x)\d x=0$. The assumption \textbf{H1} is the standard sufficient condition ensuring the stability of the target state $\bar{x}$ in probability relative to the domain $\Lambda_l$~\cite[Theorem 2.2]{mao2007stochastic}, and the hypothesis \textbf{H2} ensures the attractivity of the neighborhood of $\bar{x}$ determined by \textbf{H1}. Consequently, the solution of the switched system is non-explosive~\cite[Chapter 3.4]{khasminskii2011stochastic}. This establishes a foundation for the application of the classical hysteresis switching technique~\cite{morse1992applications}. 

Suppose that \textbf{H1} and \textbf{H2} hold. Inspired by the scale-independent hysteresis switching logic~\cite[Chapter 6]{liberzon2003switching}, for all $k\in[m]$, we define the regions
\begin{equation}
\Theta^{l^*-\epsilon}_k:=\big\{x\in \Lambda^c_{l^*-\epsilon}: \,\mathcal{A}_k V(x)< r \mathbf{A}_V(x)\big\},
\label{Eq:Region_2}
\end{equation}
where the constants $r\in(0,1)$ and $\epsilon\in(0,l^*)$ are used to control the dwell-time and the number of switches. 
Then, \textbf{A1} is satisfied, that is
$
\mathbb{R}^{d\times d}\subset\Lambda_{l}\cup\bigcup_{k\in[m]}\Theta^{l^*-\epsilon}_k.
$
Otherwise, there exists a $x\in\mathbb{R}^{d\times d}\setminus\Lambda_{l^*-\epsilon}$ such that $\mathcal{A}_k V(x)\geq r \mathbf{A}_V(x)$ for all $k\in[m]$, which leads to a contradiction since $\mathbf{A}_V(x)<0$. 

Based on \textbf{H1}, \textbf{H2} and Assumption \ref{ass:partition}, we define the following switching law $\sigma_1$.
\begin{definition}[Switching law $\sigma_1$]
For any initial state $x_0\in \mathbb{R}^{d\times d}\setminus \bar{x}$, set $\tau_0:=0$ and 
\begin{equation*}
\begin{split}
&p_0:=
\begin{cases}
\arg\min_{k\in[m]}\mathcal{A}_kV(x_0),& \text{if }x_0\in \mathbb{R}^{d\times d}\setminus \Lambda_{l^*-\epsilon};\\
\mathbf{j},& \text{if }x_0\in \Lambda_{l^*-\epsilon}\setminus \bar{x},
\end{cases}\\
&u^{p_0}_0 := \mathds{1}_{\{k=p_{0}\}}.
\end{split}
\end{equation*}
Then, set for all  $n\ge 0$
\begin{equation*}
\begin{split}
&\tau_{n+1}:=
\begin{cases}
\inf\{t\geq \tau_n:\,X_t\notin \Theta^{l^*-\epsilon}_{p_n}\},& \text{if }X_{\tau_n}\in \mathbb{R}^{d\times d}\setminus \Lambda_{l*-\epsilon};\\
\inf\{t\geq \tau_n:\,X_t\notin\Lambda_l\},& \text{if }X_{\tau_n}\in \Lambda_{l^*-\epsilon},
\end{cases}\\
&p_{n+1}:=
\begin{cases}
\arg\min_{j\in[m]}\mathcal{A}_jV(X_{\tau_{n+1}}),& \text{if }X_{\tau_{n+1}}\in \mathbb{R}^{d\times d}\setminus \Lambda_{l*-\epsilon};\\
\mathbf{j},& \text{if }X_{\tau_{n+1}}\in \Lambda_{l^*-\epsilon},
\end{cases}\\
&u^{k}_{\tau_{n+1}} := \mathds{1}_{\{k=p_{n+1}\}}, \quad \forall k\in[m].
\end{split}
\end{equation*}
Under Assumption \ref{ass:partition}, $\tau_{n+1}>\tau_n$ almost surely for all $n\ge 0$. 
\end{definition}

In the context of the system~\eqref{Eq:SDE} under the switching law $\sigma_1$, we make the following assumption for a specific scenario based on Assumption \ref{ass:partition}.
Consider the case where $\mathbf{j}=m$ and 
$\mathbb{R}^d\subset\Lambda_{l}\cup\bigcup_{k\in[m-1]}\Theta^{l^*-\epsilon}_k$.
\begin{assumption}\label{ass:partition2}
For $k\in[m-1]$, 
$$X + h_k\bigl(X,x\bigr)\in \Theta^{l^*-\epsilon}_k,\quad\forall\,|x|\le c$$ 
holds if $X\in \Theta^{l^*-\epsilon}_k$, and 
$$X + h_k\bigl(X,x\bigr)\in \Lambda_l,\quad\forall\,|x|\le c$$ 
holds if $X\in \Lambda_{l^*-\epsilon}$.    
\end{assumption}


Inspired by~\cite[Lemma 4.10]{mirrahimi2007stabilizing} and~\cite[Theorem 6.3]{liang2019exponential}, we can conclude the following stochastic analog of practical stability of the target state $\bar{x}$.  

\begin{lemma}
Consider the switched system~\eqref{Eq:SDE} under switching law $\sigma_1$. Suppose that \emph{\textbf{H1}}, \emph{\textbf{H2}} and Assumption \ref{ass:partition} (or Assumption \ref{ass:partition2}) are satisfied. For any initial state $x_0\in\mathbb{R}^d$, the switch occurs only finite times for almost each sample path, then the trajectories will stay in $\Lambda_l$ and never exit and $\mathbf{j}$-subsystem will be active afterwards almost surely.
\label{Lemma:NeverExit}
\end{lemma}
\begin{proof}
  The proof consists of the following three steps:
\begin{enumerate}
\item We show that for all $x_0\in\Lambda_{l^*-\epsilon}$, the probability of $X_t$ exiting $\Lambda_l$ is strictly less than one.
\item We show that, for all sample path that $X_t$ exits $\Lambda_l$ in finite time, $X_t$ can enter $\Lambda_{l^*-\epsilon}$ again in finite time.
\item We show that, the switch occurs only finite times for almost all sample path, then the trajectories will stay in $\Lambda_l$ and never exit and $\mathbf{j}$-subsystem will be active afterwards almost surely.
\end{enumerate}
\noindent \emph{Step 1.} Assumption \textbf{H1} implies that $V(x)\geq \mu_1(|x|)$ for all $x\in\bar{\Lambda}_l$ and 
\begin{equation*}
\mu_1(l^*-\varepsilon)\le \sup_{x\in\Lambda_{l^*-\epsilon}}V(x)\leq \mu_2(l^*-\epsilon)=\alpha \mu_1(l), \quad \text{with }\alpha= \mu_2(l^*-\epsilon)/\mu_1(l)\in(0,1),
\end{equation*}
where $l^*=\mu^{-1}_2\circ\mu_1(l)\in(0,l]$ and $\epsilon\in(0,l^*)$.
Set $x_0\in\Lambda_{l^*-\epsilon}$ and denote the first exit time from $\Lambda_l$ by $T_l:=\inf\{t\geq t_0: X_t\notin \Lambda_l \}$. Due to It\^o formula and the definition of switching law, we have 
\begin{equation*}
\mathbb{E}\big[V(X_{T_l\wedge t})\big]= V(x_0)+\mathbb{E}\left[\int^{T_l\wedge t}_{t_0} \mathcal{A}_\mathbf{j}V(X_s) \d s\right] \leq V(x_0) \leq \alpha \mu_1(l).
\end{equation*}
Moreover, 
$\mathbb{E}\big[V(x(T_l\wedge t))\big]\geq \mathbb{E}\big[\mathds{1}_{\{T_l\leq t\}} V(x(T_l)\big]\geq \mu_1(l)\mathbb{P}(T_l\leq t)$. Thus, $\mathbb{P}(T_l\leq t) \leq \alpha<1$. Letting $t\to\infty$, we obtain $\mathbb P(\exists t\ge  t_0 \mbox{ such that } X_t\notin \Lambda_l)=\mathbb{P}(T_l<\infty) \leq \alpha$. Hence,
\begin{equation*}
\mathbb{P}(X_t\in \Lambda_l,\, \forall t\geq t_0) \geq 1-\alpha>0,\quad \forall x_0\in\Lambda_{l^*-\epsilon}.
\end{equation*}
\emph{Step 2.} Define two sequences of stopping times, $\mathfrak{s}_0:= t_0$ and for all $n\geq 1$
\begin{equation*}
\mathfrak{t}_n:=\inf\{t\geq \mathfrak{s}_{n-1}:\, X_t\in \Lambda_{l^*-\epsilon}\}, \quad \mathfrak{s}_n:=\inf\{t\geq \mathfrak{t}_n:\, X_t\notin \Lambda_{l}\}.
\end{equation*}
 Due to \textbf{H2}, by applying the It\^o formula, for all $n\in\mathbb{N}$ and $t> 0$, we have
\begin{equation*}
\begin{split}
\mathbb{E}\big[V(X_{\mathfrak{t}_{n+1}\wedge t}) \big]-\mathbb{E}\big[V(X_{\mathfrak{s}_{n}\wedge t}) \big]&=\mathbb{E} \left[\int^{\mathfrak{t}_{n+1}\wedge t}_{\mathfrak{s}_{n}\wedge t} \sum^m_{k=1}u^k_s \mathcal{A}_{k}V(X_s))\d s\right]\\
&\leq -\delta \mathbb{E}[\mathfrak{t}_{n+1}\wedge t-\mathfrak{s}_{n}\wedge t].
\end{split}
\end{equation*}
Since $\mathfrak{t}_{n+1}\wedge t\geq \mathfrak{s}_{n}\wedge t$ almost surely and $V(x)$ is bounded by a constant $\mathfrak{l}>0$ for all $x\in\bar{\Lambda}_{l}$ due to the compactness and the continuity of $V$, we obtain
\begin{equation*}
\mathbb{E}\big[(\mathfrak{t}_{n+1}\wedge t-\mathfrak{s}_{n}\wedge t) \mathds{1}_{\{\mathfrak{s}_{n}< t \leq \mathfrak{t}_{n+1}\}} \big]\leq \mathbb{E}[\mathfrak{t}_{n+1}\wedge t-\mathfrak{s}_{n}\wedge t] \leq \frac{\mathfrak{l}}{\delta},
\end{equation*}
which implies 
\begin{equation*}
\mathbb{P}(\mathfrak{s}_{n}< t \leq \mathfrak{t}_{n+1})\leq \mathbb{E}\left[\frac{\mathfrak{s}_{n}\wedge t}{t}\mathds{1}_{\{\mathfrak{s}_{n}< t\}}\right] + \frac{\mathfrak{l}}{\delta t},
\end{equation*}
where $\frac{\mathfrak{s}_{n}\wedge t}{t}\mathds{1}_{\{\mathfrak{s}_{n}< t\}}<1$ almost surely. Letting $t\to\infty$, the dominated convergence theorem yields
$
\mathbb{P}(\mathfrak{t}_{n+1}=\infty,  \mathfrak{s}_{n}< \infty)=0. 
$
It implies that $\mathbb{P}(\mathfrak{t}_{n+1}<\infty|\mathfrak{s}_{n}< \infty)=1$ if  $\mathbb{P}(\mathfrak{s}_{n}< \infty)>0$. If $\mathbb{P}(\mathfrak{s}_{n}< \infty)=0$, the trajectory will never exit $\Lambda_l$.

\noindent \emph{Step 3.} If $\mathfrak{t}_{n} \leq t<\mathfrak{s}_{n}$, only $\mathbf{j}$-subsystem is active. Combining the strong Markov property\footnote{See~\cite[Proposition 3.7]{mirrahimi2007stabilizing} for the proof.} and \emph{Step 1}, we have
\begin{equation*}
\mathbb{P}(\mathfrak{s}_n<\infty|\mathfrak{t}_n<\infty)\leq \alpha<1.
\end{equation*}
By similar arguments as in~\cite[Lemma 4.10]{mirrahimi2007stabilizing}, and using Bayes' formula and the Borel-Cantelli lemma, we have
\begin{equation*}
\mathbb{P}(\mathfrak{s}_n<\infty \text{ for infinitely many }n)=0.
\end{equation*}
Thus, for almost each sample path, there exists an integer $M<\infty$ such that $\mathfrak{s}_n=\infty$ for all $n\geq M$, and $\mathfrak{s}_n<\infty$ for all $n<M$. Hence, for almost all sample paths, there are only finite switches between $\Lambda_l$ and $\{\Theta^{l^*-\epsilon}_k\}_{k\in[m]}$ and the $\mathbf{j}$-subsystem will be always active afterwards. Moreover, due to the non-empty overlap of each adjacent open region $\Theta^{l^*-\epsilon}_k$ with $k\in [m]$, only finite switches occur between $\Theta^{l^*-\epsilon}_k$. The proof is thus fulfilled.
\end{proof}

The main result of this section can be stated below.
\begin{theorem}
Suppose that \emph{\textbf{H1}}, \emph{\textbf{H2}} and  Assumption \ref{ass:partition} (or Assumption \ref{ass:partition2})  are satisfied. Then, for the switched system~\eqref{Eq:SDE} under the switching law $\sigma_1$, the switch occurs only finite times for almost all sample path, and the target state $\bar{x}$ is GAS in mean and almost surely.
\label{Theorem:GAS}
\end{theorem}

\begin{proof}
For the stochastic differential equation corresponding to the $\mathbf{j}$-subsystem, the solution is a strong Markov process, Feller continuous and stochastically continuous uniformly in $t$ and initial state~\cite{protter2004stochastic}. Define the event 
$$
\Omega_l:=\{\omega\in\Omega:\, X_t \text{ never exits }\Lambda_l \text{ and }\mathbf{j}\text{-subsystem is active}\},
$$
where $X$ is the solution of \eqref{Eq:SDE} under $\sigma$. 
Together with \textbf{H1}, all conditions of the stochastic LaSalle invariance theorem~\ref{Thm:LaSalle} are satisfied. Thus, for almost all $\omega\in\Omega_l$, $X_t$ converges in probability to $\bar{x}$ when $t\to\infty$. It implies
$
\lim_{t\rightarrow \infty}\mathbb{P}\big(\mathsf{V}(X_t)>\varepsilon\big|\Omega_l \big)=0, 
$ for all $\varepsilon>0$, where $\mathsf{V}$ is defined in \textbf{H1}. Due to the continuity of $\mathsf{V}(x)$ and the compactness of $\bar{\Lambda}_l$, there exists a constant $\mathfrak{l}>0$ such that $\mathsf{V}(x)\leq c$ for all $x\in\bar{\Lambda}_l$. Then, we deduce
\begin{equation*}
\mathbb{E}\big[\mathsf{V}(X_t)\big|\Omega_l\big]\leq \mu_2(l)\mathbb{P}\big(\mathsf{V}(x(t))>\varepsilon\big|\Omega_l \big)+\varepsilon\big(1-\mathbb{P}\big(\mathsf{V}(x(t))>\varepsilon\big|\Omega_l \big)\big),
\end{equation*}
which implies $\lim_{t\rightarrow \infty}\mathbb{E}\big[\mathsf{V}(x(t))\big|\Omega_l]=0$. By virtue of Theorem~\ref{Thm:LaSalle}, $\mathsf{V}(X_t)$ converges for almost all $\omega\in\Omega_l$. By using the dominated convergence theorem, we get $\mathbb{E}\big[\lim_{t\rightarrow \infty}\mathsf{V}(X_t)\big|\Omega_l \big]=0$. Then, $X_t$ converges to $\bar{x}$ when $t$ tends to infinity for almost all $\omega\in\Omega_l$. Moreover, due to Lemma~\ref{Lemma:NeverExit}, we have $\mathbb{P}(\Omega_l)=1$. Thus, $\lim_{t\to\infty}X_t=\bar{x}$  almost surely. In addition, the stability of $\bar{x}$ in probability is proved in \emph{Step~1} of Lemma~\ref{Lemma:NeverExit}, which completes the proof.
\end{proof}

Inspired by~\cite{liang2019exponential,applebaum2009asymptotic}, in Proposition~\ref{Proposition:GES}, we provide sufficient conditions ensuring almost sure GES of the target state $\bar{x}$ for the switched system~\eqref{Eq:SDE} under $\sigma$. To establish the groundwork for our results, we introduce the following assumption, ensuring that the target state is non-attainable in finite time almost surely.
\begin{assumption}\label{ass:partition3}
There exists a constant $\delta>0$ and $\mathbf{j}\in[m]$ such that $|x+h_{\mathbf{j}}(x,z)|\geq \delta |x|$ for all $x\in \Lambda_{l}$ and $|z|\leq c$.    
\end{assumption}
By applying the similar arguments as in\cite[Lemma 3.2]{applebaum2009asymptotic} and~\cite[Lemma 2.6]{chao2017almost}, we derive the following lemma, which enables us to work with the functions that are twice continuously differentiable in any neighbourhood of the target state $\bar{x}$.
\begin{lemma}\label{Lemma:NeverReach}
    Assume that \emph{\textbf{H1}} and Assumption \ref{ass:partition3} are satisfied. Then, for the $\mathbf{j}$-subsystem, the related trajectory $X_t$ cannot attain the target state $\bar{x}$ in finite time almost surely, i.e.,
    \begin{equation*}
        \mathbb{P}\big(X_t\neq \bar{x},\, \forall t\geq t_0\big)=1, \quad \forall x_0\neq \bar{x}.
    \end{equation*}
\end{lemma}
The following result addresses the GES of the target state.
\begin{proposition}
Consider the switched system~\eqref{Eq:SDE} under $\sigma$. 
There exists $\mathsf{V}\in\mathcal{C}^2(\mathbb{R}^{d\times d},\mathbb{R}_+)$ such that $\mathsf{V}(x)=0$ if and only if $x=\bar{x}$. Moreover, suppose that there exist $\mathbf{j}\in[m]$,   $c_1,c_2,c_3,l\in (0,\infty)$ and  $c_4,c_5\in\R_+$ and $\mu\in \mathcal{K}$ such that, for all $x\in\bar{\Lambda}_l$, 
\begin{enumerate}
\item[\emph{(i)}] $c_1|x-\bar{x}|^{c_2}\leq \mathsf{V}(x)\leq \mu(|x|)$,
\item[\emph{(ii)}] $\mathcal{A}_{\mathbf{j}}\mathsf{V}(x)\leq -c_3 \mathsf{V}(x)$,
\item[\emph{(iii)}] $\liminf_{x\rightarrow\bar{x}}|\mathfrak{D}_x \mathsf{V}\cdot g_{\mathbf{j}}(x)|^2/\mathsf{V}(x)^2\geq c_4 $,
\item[\emph{(iv)}] $\limsup_{x\rightarrow\bar{x}}\int_{|z|\leq c}\mathfrak{V}_{\mathbf{j}}(x,z)\nu(dz)\leq -c_5$,
\end{enumerate}
where $\mathfrak{V}_{\mathbf{j}}(x,z)=\log\Big(\frac{\mathsf{V}\big(x+h_{\mathbf{j}}(x,z)\big)}{\mathsf{V}(x)}\Big)+1-\frac{\mathsf{V}\big(x+h_{\mathbf{j}}(x,z)\big)}{\mathsf{V}(x)}$.
Additionally, assume that for all bounded sets $\mathcal{M}\subset\mathbb{R}^d$,  
\begin{equation}\label{Eq:Cond_h}
    \sup_{x\in\mathcal{M}}\sup_{0<|z|<c}|h_{\mathbf{j}}(x,z)|<\infty,
\end{equation}
and suppose that \emph{\textbf{H2}} with $l^*=\mu^{-1}(c_1 l^{c_2})$, Assumption \ref{ass:partition} (or Assumption \ref{ass:partition2})  and \ref{ass:partition3} hold true. Then, $\bar{x}$ is GES almost surely with the sample Lyapunov exponent less than or equal to $-(2c_3+c_4+2c_5)/2c_2$.
\label{Proposition:GES}
\end{proposition}
\begin{proof}
 The condition (i) and (ii) implies that \textbf{H1} is satisfied. By applying Theorem~\ref{Theorem:GAS}, for the system~\eqref{Eq:SDE} under switching law $\sigma_1$, the solution is non-explosive and the target state $\bar{x}$ is GAS in mean and almost surely. Next, we estimate the convergence rate of $x(t)$ toward $\bar{x}$.

Let $\bar{V}(x)$ be a twice continuously differentiable positive definite function on $\mathbb{R}^{d\times d}\setminus \bar{x}$ which is equal to $\mathsf{V}(x)$ for all $x\in\Lambda_{l}$ and $c_1|x-\bar{x}|^{c_2}\leq \bar{V}(x) \leq \varphi(|x-\bar{x}|)$ for all $x\in\mathbb{R}^{d\times d}$ and for some $\varphi\in\mathcal{K}$.
It is straightforward to compute that, for all $k\in [m]$, we have
\begin{align*}
        \mathcal{A}_k \log \bar{V}(x)=&\frac{\mathfrak{D}_x\bar{V} f_k(x)}{\bar{V}(x)}+\frac{1}{2\bar{V}(x)}\mathrm{Tr}\Big(g^*_k(x)  \mathfrak{D}^2_x \bar{V}  g_k(x)\Big)-\frac{|\mathfrak{D}_x\bar{V}g_k(x)|^2}{2\bar{V}(x)^2}\\
        &+\int_{|z|\leq c}\Big(\log \frac{\bar{V}\big(x+h_k(x,z)\big)}{\bar{V}(x)}- \frac{\mathfrak{D}_x\bar{V} h_k(x,z) }{\bar{V}(x)}\Big) \nu(dz)\\
        =& \frac{\mathcal{A}_k \bar{V}(x)}{\bar{V}(x)}-\frac{|\mathfrak{D}_x\bar{V}g_k(x)|^2}{2\bar{V}(x)^2}+\int_{|z|\leq c}\bar{\mathfrak{V}}_{k}(x,z) \nu(dz),
\end{align*}
where $\bar{\mathfrak{V}}_{k}(x,z)=\log \frac{\bar{V}\big(x+h_k(x,z)\big)}{\bar{V}(x)}+1- \frac{\bar{V}\big(x+h_k(x,z)\big) }{\bar{V}(x)}$.
Due to Lemma~\ref{Lemma:NeverReach}, the target state is non-attainable in finite time almost surely.   Then, for all $x_0\in\mathbb{R}^{d\times d}\setminus \bar{x}$, $\bar{V}(X_t)>0$ for all $t\geq t_0$ almost surely.
By applying It\^o formula, for all $t\geq t_0$,
\begin{equation}\label{Eq:LogV}
\log \bar{V}(X_t)=\log \bar{V}(x_0)+\int^{t}_{t_0} \sum^m_{k=1}u^k_{s} \mathcal{A}_k \log \bar{V}(X_s)\d s+M_t,
\end{equation}  
where $M_t=M^1_t+M^2_t$, and
\begin{align*}
    M^1_t&:=\int^{t}_{t_0}  \sum^m_{k=1}u^k_{s} \frac{\mathfrak{D}_x\bar{V} g_k(X_s) }{\bar{V}(X_s)}\d W^k_s, \\
    M_t^2&:=\int^{t}_{t_0} \int_{|z|\leq c}\sum^m_{k=1}u^k_{s}\log \frac{\bar{V}\big(X_{s-}+h_k(X_{s-},z)\big)}{\bar{V}(X_{s-})}  \tilde{N}^{k}(\d s,\d z),
\end{align*}
are martingales vanishing at $t=t_0$.
Note that, due to the assumption \textbf{A2} and the condition~\eqref{Eq:Cond_h}, by slightly modifying the proof of~\cite[Lemma 3.3]{applebaum2009asymptotic},  the following integral in Equation~\eqref{Eq:LogV} satisfies
\begin{align*}
    \left|\int^t_{t_0}\int_{|z|\leq c} \sum^m_{k=1}u^k_{s}\bar{\mathfrak{V}}_{k}(x(s-),z) \d z\d s \right|<\infty.
\end{align*}
By the exponential martingale inequality with jumps~\cite[Theorem 5.2.9]{applebaum2009levy}, for any integer $\mathfrak{n}\geq t_0$, $\theta\in(0,1)$,
\begin{align*}
    \mathbb{P}\Big[\sup_{t_0\leq t\leq \mathfrak{n}}\Big(M_t-\frac{\theta}{2}\langle M^1, M^1\rangle_t-\mathfrak{f}_{\theta}(t) \Big)>\theta \mathfrak{n}   \Big]\leq e^{-\theta^2\mathfrak{n}},
\end{align*}
where
\begin{align*}
&\langle M_1, M_1\rangle_t := \int^t_{t_0}\sum^m_{k=1}u^k_{s} \frac{|\mathfrak{D}_x\bar{V} g_k(X_s)|^2}{\bar{V}^2(X_s)} \d s,\\
    &\mathfrak{f}_{\theta}(t)=\frac{1}{\theta}\int^t_{t_0}\int_{|z|\leq c}\sum^m_{k=1}u^k_{s}\Big[ \Big(\frac{\bar{V}\big(X_{s-}+h_k(X_{s-},z)\big)}{\bar{V}(X_{s-})}\Big)^{\theta}-1- \theta\log \frac{\bar{V}\big(X_{s-}+h_k(X_{s-},z)\big) }{\bar{V}(X_{s-})}\Big] \d z\d s.
\end{align*}
Since $\sum^{\infty}_{\mathfrak{n}=1}e^{-\theta^2\mathfrak{n}}<\infty$, by Borel-Cantelli lemma we have that for almost all sample paths there exists $\mathfrak{n}_0=\mathfrak{n}_0(\omega)$ such that, for all $\mathfrak{n}\geq \mathfrak{n}_0$ and $t_0\leq t\leq \mathfrak{n}$,
\begin{align*}
    M_t\leq \frac{\theta}{2}\langle M^1, M^1\rangle_t+\theta \mathfrak{n}+\mathfrak{f}_{\theta}(t).
\end{align*}
It follows that, for all $\mathfrak{n}\geq \mathfrak{n}_0$ and $t_0\leq t\leq \mathfrak{n}$, 
\begin{align*}
    \log \bar{V}(X_t)\leq& \log \bar{V}(x_0)+\int^{t}_{t_0} \sum^m_{k=1}u^k_{s} \frac{\mathcal{A}_k \bar{V}(X_s)}{\bar{V}(X_s)}\d s-\frac{1-\theta}{2}\langle M^1, M^1\rangle_t+\theta \mathfrak{n}+\mathfrak{f}_{\theta}(t)\\
    &+\int^t_{t_0}\int_{|z|\leq c}\sum^m_{k=1}u^k_{s}\bar{\mathfrak{V}}_{k}(X_{s-},z) \d z\d s.
\end{align*}
By following the same arguments as in the proof of~\cite[Theorem 3.1]{applebaum2009asymptotic}, for all $t\geq t_0$, we have $\lim_{\theta\rightarrow 0}\mathfrak{f}_{\theta}(t)=0$. Letting $\theta\to 0$, we have
\begin{align*}
    \limsup_{t\rightarrow\infty} \frac{1}{t}\log \bar{V}(X_t)\leq  \limsup_{t\rightarrow\infty} \frac{1}{t} \Big[& \int^{t}_{t_0} \sum^m_{k=1}u^k_{s} \frac{\mathcal{A}_k \bar{V}(X_s)}{\bar{V}(X_s)}\d s-\frac{1}{2}\langle M^1, M^1\rangle_t\\
    &+\int^t_{t_0}\int_{|z|\leq c}\sum^m_{k=1}u^k_{s}\bar{\mathfrak{V}}_{k}(X_{s-},z) \d z\d s\Big], \quad a.s.
\end{align*}
For every fixed constant $T\geq t_0$ and an arbitrary $\mathsf{r}\in(0,l^*-\epsilon)$, consider the event $\Omega^{\mathsf{r}}_T:=\{\omega\in\Omega:\,X_t\in\Lambda_{\mathsf{r}},\, t\geq T\}$. Based on the definition of $\bar{V}$ and the switching law $\sigma_1$ and the condition (ii), for almost all $\omega\in\Omega^{\mathsf{r}}_T$,  
\begin{align*}
\limsup_{t\rightarrow\infty}& \frac{1}{t} \Big[ \int^{t}_{t_0} \sum^m_{k=1}u^k_{s} \frac{\mathcal{A}_k \bar{V}(X_s)}{\bar{V}(X_s)}\d s-\frac{1}{2}\langle M^1, M^1\rangle_t+\int^t_{t_0}\int_{|z|\leq c}\sum^m_{k=1}u^k_{s}\bar{\mathfrak{V}}_{k}(X_{s-},z) \d z\d s\Big]\\
&\leq \limsup_{t\rightarrow\infty} \frac{1}{t} \Big[ \int^{t}_{T}  \frac{\mathcal{A}_{\mathbf{j}} \mathsf{V}(X_s)}{\mathsf{V}(X_s)}\d s-\int^{t}_{T}\frac{|\mathfrak{D}_x\mathsf{V}g_{\mathbf{j}}(x)|^2}{2\mathsf{V}(x)^2}ds+\int^t_{T}\int_{|z|\leq c}\mathfrak{V}_{\mathbf{j}}(X_{s-},z) \d z\d s\Big]\\
&\leq -c_3-\inf_{x\in \Lambda_{\mathsf{r}}\setminus \bar{x}}\frac{|\mathfrak{D}_x\mathsf{V}g_{\mathbf{j}}(x)|^2}{2\mathsf{V}(x)^2} + \sup_{x\in \Lambda_{\mathsf{r}}\setminus \bar{x}}\int_{|z|\leq c}\mathfrak{V}_{\mathbf{j}}(x,z) \d z.
\end{align*}
For arbitrary small $\mathsf{r}>0$, Theorem~\ref{Theorem:GAS} implies  $\lim_{T\rightarrow\infty}\mathbb{P}(\Omega^{\mathsf{r}}_T)=1$. Therefore, due to conditions (i), (iii) and (iv), we have
\begin{equation*}
    \limsup_{t\rightarrow \infty}\frac{1}{t}\log |X_t-\bar{x}|\leq -\frac{2c_3+c_4+2c_5}{2c_2}, \quad a.s.
\end{equation*}
which yields the result.   
\end{proof}

\section{Hysteresis switching for quantum stabilization}

The potential of dissipative switching control has been explored in recent works~\cite{liang2024switching,grigoletto2021stabilization,scaramuzza2015switching}, extending beyond coherent control to include dissipative resources. This involves utilizing measurements and controlled interactions with engineered fields and quantum environments. The hysteresis switching method introduced in~\cite[Section 3]{liang2024switching} assumes that the target state or subspace remains invariant under all controlled dynamics. While this aligns with classical switched systems~\cite{liberzon2003switching,teel2014stability},  it presents a challenge in designing coupled fields. To address this,~\cite[Section 4]{liang2024switching} proposed a novel design framework that modulates the intensity of the switched generators, thereby relaxing the strict invariance requirement. However, implementing this strategy in practice remains challenging due to the need for precise control of interaction strengths.

In this section,
we consider quantum systems described on $\H$ and turn to \eqref{eq:sme_switch}, where the monitored system can be coupled to one of a finite set of external systems during assigned period of times. The effect of these couplings on the dynamics are our control resources, and which is active at which time is going to be determined by a switching law. Assuming that these external systems act as memory-less (Markov) environments~\cite{alicki2007quantum}, and only one external system is coupled to the target one at any given time. We will apply the hysteresis switching strategy from Section \ref{Sec:Hysteresis} for quantum stabilization. Notably, this strategy relaxes the invariance condition to just one switched generator and ensures only a finite number of switches occur. This simplification enhances practical applicability.



\subsection{Invariant and stable subspaces} 
Let $\H_S\subset \H$ be the target subspace. Denote by $\Pi_{0}\notin\{0,\mathds{1}\}$ the orthogonal projection on $\H_S\subset \H$. Define the set of density matrices
\begin{equation*}
\mathcal{I}(\H_S):=\{\rho\in\mathcal{S}(\H): \mathrm{Tr}(\Pi_{0}\rho)=1\},
\end{equation*}
namely those whose support is contained in $\H_S$.
\begin{definition}
For the switched system~\eqref{eq:sme_switch}, the subspace $\H_S$ is called invariant almost surely if $\rho_0\in \mathcal{I}(\H_S)$, $\rho_t\in \mathcal{I}(\H_S)$ for all $t>0$ almost surely.
\end{definition}

Let $\H=\H_S\oplus\H_R$ and $X\in\mathcal{B}(\H)$, the matrix representation in an appropriately chosen basis can be written as 
\begin{equation*}
X=\left[\begin{matrix}
X_S & X_P\\
X_Q & X_R
\end{matrix}\right],
\end{equation*}
where $X_S,X_R,X_P$ and $X_Q$ are matrices representing operators from $\H_S$ to $\H_S$, from $\H_R$ to $\H_R$, from $\H_R$ to $\H_S$, from $\H_S$ to $\H_R$, respectively.
The invariance property of \eqref{eq:sme0} corresponds directly to the structure of $\mathcal{F}_k$. 
\begin{theorem}[{\cite{ticozzi2008quantum}}]
For the system~\eqref{eq:sme0}, the subspace $\H_S\subset\H$ is 
invariant almost surely if and only if $L_{k,Q}=C_{k,Q}=D_{k,Q}=0$ and $\i(H_{0,P}+H_{k,P})=\frac{1}{2}(L^*_{k,S}L_{k,P}+C^*_{k,S}C_{k,P}+D^*_{k,S}D_{k,P})$.
\label{Thm:Invariance}
\end{theorem}

Based on the block-decomposition with respect to the orthogonal direct sum decomposition $\H=\H_S\oplus\H_R$, for any  $X_R\in\mathcal{B}(\H_R)$, we call the {\em extension of $X_R$ to} $\mathcal{B}(\H)$ the following matrix:
\begin{equation*}
X=\left[\begin{matrix}
0 & 0\\
0 & X_R
\end{matrix}\right].
\end{equation*} 
In order to quantify the distance between $\rho\in\mathcal{S}(\H)$ and $\mathcal{I}(\H_S)$ we shall make use of linear functions associated to a  positive $X_R\in\mathcal{B}_{+}(\H_R)$, namely
$$
\mathrm{Tr}(X \rho)=\mathrm{Tr}(X_R \rho_R)\in[0,1],
$$
where $X$ is the extension in $\mathcal{B}(\H)$ of $X_R$. Such function is used as an estimation of the distance $\mathbf{d}_{0}(\rho)$. 
\begin{lemma}[{\cite{liang2024switching}}]
For all $\rho\in\mathcal{S}(\H)$ and the orthogonal projection $\Pi_{0}\in\mathcal{B}(\H)$ on $\mathcal{H}_S$, there exist two constants $c_1>0$ and $c_2>0$ such that 
\begin{equation}
c_1\mathrm{Tr}(X \rho)\leq \|\rho-\Pi_{0}\rho \Pi_{0}\| \leq c_2 \sqrt{\mathrm{Tr}(X \rho)},
\label{Eq:Relation_DisLya}
\end{equation}
where $X$ is the extension in $\mathcal{B}(\H)$ of $X_R$.
\label{Lemma:Relation_DisLya}
\end{lemma}

To study the stabilization of the quantum system under switched SME dynamics, we extend the classical notions of stochastic stability (Definition~\ref{Def:Stability}) to the context of density operators and invariant subspaces. Inspired by~\cite{ticozzi2008quantum,benoist2017exponential}, we define the following measure of convergence to a target subspace in terms of the projection distance. In the following definition, denote $\mathbf{d}_{0}(\rho):=\|\rho-\Pi_{0}\rho\Pi_{0}\|$ where $\|\cdot\|$ could be any matrix norm.

\begin{definition}
The subspace $\H_S\subset \mathbb{H}$ is said to be
\begin{enumerate}
\item
\emph{stable in probability}, if for every pair $\varepsilon \in (0,1)$ and $r>0$, there exists $\delta = \delta(\varepsilon,r,t_0)>0$ such that,
\begin{equation*}
\mathbb{P} \big( \mathbf{d}_{0}(\rho_t)< r \text{ for } t \geq 0 \big) \geq 1-\varepsilon,
\end{equation*}
whenever $\mathbf{d}_{0}(\rho_0)<\delta$. 

\item
almost surely GAS, if it is stable in probability and,
\begin{equation*}
\mathbb{P} \left( \lim_{t\rightarrow\infty}\mathbf{d}_{0}(\rho_t)=0 \right) = 1, \quad \forall \rho_0\in\mathcal{S}(\H).
\end{equation*}

\item almost surely GES, if
\begin{equation*}
\P\left(\limsup_{t \rightarrow \infty} \frac{1}{t} \log \big( \mathbf{d}_{0}(\rho_t)\big) < 0\right)=1,\quad \forall \rho_0\in\mathcal{S}(\H).
\end{equation*}
\end{enumerate}
\end{definition}
The control problem under consideration is the following. Given a target subspace $\H_S\subset \H$, construct switching laws $u_t$ that admits a set of non-Zeno and non-chattering switching sequence, which ensures that $\H_S$ is GAS and/or GES almost surely.

\subsection{Measurement-dependent switching strategies}
In~\cite[Section III]{liang2024switching}, we proposed four switching strategies ensuring the GAS/GES of the target subspace $\H_S$ under the following assumptions:
\begin{enumerate}
    \item $\H_S$ is invariant with respect to $\mathcal{F}_{k}$ for all $k\in[m]$.
    \item There exist $K_R\in\mathcal{B}_{+}(\H_R)$ and a constant $c>0$ such that 
for all $\rho\in\mathcal{S}(\H)$,
$\mathbf{L}_K(\rho)\leq -c\mathrm{Tr}(K\rho)$, where $K$ is the extension in $\mathcal{B}(\H)$ of $K_R$ and $\mathbf{L}_K(\rho):=\min_{k\in[m]}\mathrm{Tr}(K\mathcal{F}_k(\rho))$.
\end{enumerate}
However, the above assumptions are maybe difficult to realize in the real world situation. Moreover, from the simulations in~\cite[Section V]{liang2024switching}, the estimation of the Lyapunov exponent provided by~\cite[Section III]{liang2024switching} is too rough since we use a linear common Lyapunov function the diffusion term of SME is ignored. In order to relax the above assumptions, in~\cite[Section III]{liang2024switching}, we supposed the gains of all Hamiltonians and noise operators are adjustable and proposed a switching strategy ensuring the GAS of  $\H_S$ under the following assumption:
\begin{itemize}
\item There exists a $K_R\in\mathcal{B}_{+}(\H_R)$ such that $\mathbf{L}_K(\rho)<0$ for all $\rho\in\mathcal{S}(\H)\setminus\mathcal{I}(\H_S)$, where $K$ is the extension in $\mathcal{B}(\H)$ of $K_R$.
\end{itemize}
From the point of view of physical experiment operation, it is difficult to modulate the the gains of all Hamiltonians and noise operators.

In the following, based on Section~\ref{Sec:Hysteresis}, we propose a switching algorithm ensuring the GAS/GES of the target subspace $\H_S$, where the switching stops in finite time. This requires weaker assumptions on the structure of control Hamiltonian and noise operators, and we provide an approach to obtain a more precise estimation of Lyapunov exponent. Following the approach as in~\ref{sec:switchingrules_SDE}, we make the following assumption adapted to the SME.
Define $\Lambda_d:=\{\rho\in\mathcal{S}(\H):\mathbf{d}_0(\rho)< d\}$ with $d>0$ and $\bar{\Lambda}_d$ the closure of $\Lambda_d$.
\begin{itemize}
\item[$\mathbf{\overline{H}1}$:] There exist three functions $\mu_1,\mu_2,\nu\in\mathcal{K}$, and a positive definite continuous function $V(\rho)$ and a constant $l>0$ and $\mathbf{j}\in[m]$ such that $\mu_1(\mathbf{d}_0(\rho))\leq V(\rho)\leq \mu_2(\mathbf{d}_0(\rho))$ and $\mathcal{A}_{\mathbf{j}}V(\rho)\leq -\nu(\mathbf{d}_0(\rho))$ for all $\rho\in\bar{\Lambda}_l$.
\item[$\mathbf{\overline{H}2}$:] There exist $K_R\in\mathcal{B}_{+}(\H_R)$ such that $\mathbf{L}_K(\rho)<0$ for all $\rho\in\mathcal{S}(\H)\setminus \Lambda_{l^*-\epsilon}$ with $l^*:=\mu_2^{-1}\circ\mu_1(l)\in(0,l]$ and $\epsilon\in(0,l^*)$, where $K$ is the extension in $\mathcal{B}(\H)$ of $K_R$.
\end{itemize}
It is straightforward to deduce that $\H_S$ is invariant for $\mathbf{j}$-th subsystem if $\mathbf{\overline{H}1}$ is satisfied. Therefore, comparing to the assumptions proposed in the previous work, the control hypotheses $\mathbf{\overline{H}1}$ and $\mathbf{\overline{H}2}$ are of higher reliability, since they only require that $\H_S$ is invariant for $\mathbf{j}$-th subsystem.  
\begin{remark}
    The compactness of $\mathcal{S}(\H)\setminus\Lambda_{l^*-\epsilon}$, $\mathbf{\overline{H}2}$ and the continuity of $\mathbf{L}_K(\rho)$ on $\mathcal{S}(\H)$~\cite[Lemma A.2]{liang2024switching} implies that, there exists a constant $\delta>0$ such that $\mathbf{L}_K(\rho)\leq -\delta$ for all $\rho\in\mathcal{S}(\H)\setminus \Lambda_{l^*-\epsilon}$.
\end{remark}

Suppose that $\mathbf{\overline{H}1}$ and $\mathbf{\overline{H}2}$ holds. Follow the analysis in Section~\ref{Sec:Hysteresis}, for all $k\in[m]$, we define the regions
\begin{equation}
\Theta^{d^*-\epsilon}_j:=\big\{\rho\in\mathcal{S}(\H)\setminus\Lambda_{l^*-\epsilon}:\,\Tr(K\mathcal{F}_j(\rho))< r \mathbf{L}_K(\rho) \big\},
\label{Eq:Region_2}
\end{equation}
where the constants $r\in(0,1)$ and $\epsilon\in(0,d^*)$ are used to control the dwell-time and the number of switches. 
Then, we have 
$
\Lambda_{d}\bigcup\cup_{j\in[m]}\Theta^{d^*-\epsilon}_j=\mathcal{S}(\H).
$
Otherwise, there exists a $\rho\in\mathcal{S}(\H)\setminus\Lambda_{d^*-\epsilon}$ such that $\Tr(K\mathcal{F}_j(\rho))\geq r \mathbf{L}_K(\rho)$ for all $j\in[m]$, which leads to a contradiction since $\mathbf{L}_K(\rho)<0$. 

Next, in order to ensure that the switching surfaces can only be hit by the continuous part of the trajectory $\rho(t)$, we make the following assumption by modifying Assumption \ref{ass:partition2} to adapt to SME~\eqref{eq:sme_switch}, and switching law $\sigma_2$.
\begin{assumption}\label{ass:partition3}
Set $\mathbf{j}=m$. For all $k\in[m-1]$ such that $D_k\neq 0$, if $\rho\in \Theta^{l^*-\epsilon}_k$ then $\mathcal{J}_{D_k}(\rho)\in \Theta^{l^*-\epsilon}_k$, i.e., $\Tr(K\mathcal{F}_k(D_k\rho D^*_k))< r \mathbf{L}_K(D_k\rho D^*_k)$; and $\mathcal{J}_{D_\mathbf{j}}(\rho)\in \Lambda_l$ for all $\rho\in \Lambda_{l^*-\epsilon}$ with $D_{\mathbf{j}}\neq 0$. 
\end{assumption}

\begin{definition}[Switching law $\sigma_2$]
For any initial state $\rho_0\in \mathcal{S}(\H)\setminus \mathcal{I}(\H_S)$, set $\tau_0=0$ and 
\begin{equation*}
\begin{split}
&p_0:=
\begin{cases}
\arg\min_{j\in[m]}\Tr(K\mathcal{F}_j(\rho_0)),& \text{if }\rho_0\in \mathcal{S}(\H)\setminus \Lambda_{l^*-\epsilon};\\
\mathbf{j},& \text{if }\rho_0\in \Lambda_{l^*-\epsilon}\setminus \mathcal{I}(\H_S),
\end{cases}\\
&u^{k}_0 := \mathds{1}_{\{k=p_{0}\}}.
\end{split}
\end{equation*}
Then, for all $n\in\mathbb{N}$, set
\begin{equation*}
\begin{split}
&\tau_{n+1}:=
\begin{cases}
\inf\{t\geq \tau_n:\,\rho_t\notin \Theta^{d^*-\epsilon}_{p_n}\},& \text{if }\rho_{\tau_n}\in \mathcal{S}(\H)\setminus \Lambda_{l*-\epsilon};\\
\inf\{t\geq \tau_n:\,\rho_t\notin\Lambda_d\},& \text{if }\rho_{\tau_n}\in \Lambda_{l^*-\epsilon},
\end{cases}\\
&p_{n+1}:=
\begin{cases}
\arg\min_{j\in[m]}\Tr(K\mathcal{F}_j(\rho_{\tau_{n+1}})),& \text{if }\rho_{\tau_{n+1}}\in \mathcal{S}(\H)\setminus \Lambda_{l^*-\epsilon};\\
\mathbf{j},& \text{if }\rho_{\tau_{n+1}}\in \Lambda_{l^*-\epsilon},
\end{cases}\\
&u^{k}_{\tau_{n+1}} := \mathds{1}_{\{k=p_{n+1}\}}.
\end{split}
\end{equation*}
Since the overlap of each adjacent open regions, $\tau_{n+1}>\tau_n$ almost surely for all $n\in \mathbb{N}$. 
\end{definition}
 
By applying the similar arguments as in Proposition~\ref{Proposition:GES},  the main result of this section can be stated as follows:
\begin{proposition}
   Suppose that $\mathbf{\overline{H}1}$, $\mathbf{\overline{H}2}$ and Assumption~\ref{ass:partition3} are satisfied. Then, for the switched system~\eqref{eq:sme_switch} under the switching algorithm $\sigma_2$, the switch occurs only finite times for almost all sample path, and $\H_S$ is GAS in mean and almost surely.
\label{Prop:GAS_SME} 
\end{proposition}

\medskip

Comparing to~\cite[Theorem 8]{liang2024switching} under Assumptions 1 and 2, the advantage of Proposition~\ref{Prop:GAS_SME} is requiring $\mathbf{\overline{H}1}$ ensuring the existence of a local strict Lyapunov-like function for only one subsystem rather than the invariance properties of all subsystems, and requiring $\mathbf{\overline{H}2}$ which relaxes Assumption 2. However, constructing strict Lyapunov-like functions is a challenging problem even a local one. In the following Proposition~\ref{Lemma:GES_structure}, we investigate our strategies to achieve $\mathbf{\overline{H}1}$ by constructing a exponential-type Lyapunov-like function based on the structure of $\mathcal{F}_{\mathbf{j}}$. 

In fact, the candidate Lyapunov-like function $V(\rho)$ in $\mathbf{\overline{H}1}$ usually cannot be constructed as a linear function. It is difficult to construct a linear $V(\rho)$ which satisfies the exponential-type inequality $V(\mathcal{F}(\rho))\leq -cV(\rho)$ only locally not globally based on the arguments in~\cite{benoist2017exponential}. However, if $V(\rho)$ is non-linear, based on the generator~\eqref{Eq:GeneratorSME}, we have,
\begin{align}
       \mathcal{A}_{\mathbf{j}}V(\rho)= \Tr\big(\mathfrak{D}_{\rho}V \mathcal{F}_{\mathbf{j}}(\rho)\big)&+\frac{1}{2}\Tr\big(\mathfrak{D}^2_{\rho}V\,\mathcal{G}_{\mathbf{j}}(\rho)^2\big)\nonumber\\
       &+\Big( V\big(\mathcal{J}_{D_\mathbf{j}}(\rho)\big)-V(\rho)-\Tr\big(\mathfrak{D}_{\rho}V \mathcal{H}_{\mathbf{j}}(\rho)\big)\Big)V_{D_{\mathbf{j}}}(\rho).\label{Eq:InfGen}
\end{align}
 Even if $V(\rho)$ does not satisfy $\Tr\big(\mathfrak{D}_{\rho}V \mathcal{F}_{\mathbf{j}}(\rho)\big)\leq -cV(\rho)$, by suitably choosing noise operators $C_{\mathbf{j}}$ and/or $D_{\mathbf{j}}$, the second and third terms of the right-hand side of~\eqref{Eq:InfGen} may compensate the first term in order to realize $\mathcal{A}_{\mathbf{j}}V(\rho)\leq -cV(\rho)$ for some $c>0$. 
Inspired by Lyapunov functions proposed in~\cite{liang2019exponential,liang2022GHZ}, we provide the following sufficient conditions ensuring that $\mathbf{\overline{H}1}$ are satisfied by constructing a exponential-type Lyapunov-like function. 

Define the following map on $\mathcal{B}(\mathbb{H}_R)$ for $k\in[m]$,
$$
\mathcal{F}_{k,R}(\varrho):=-i[H_{k,R},\varrho]+\mathsf{I}_{L_{k}}(\varrho)+\mathsf{I}_{C_{k}}(\varrho)+\mathsf{I}_{D_{k}}(\varrho)
$$
where
 $\mathsf{I}_{X}(\varrho):=X_{R}\varrho X^*_{R}-\frac{1}{2}\{X^*_{P}X_{P}+X^*_{R}X_{R},\varrho\}$.
Denote $$\bar{\mathfrak{l}}_{k}(X_R):=\inf\{\lambda\in\mathbb{R}:\mathcal{F}^*_{k,R}(X_R)\leq \lambda X_R\}$$ where $\mathcal{F}^*_{k,R}$ is the adjoint of $\mathcal{F}_{k,R}$ with respect to Hilbert-Schmidt inner product on $\mathcal{B}(\H_R)$. Define $\mathcal{E}_{k,\delta}(X):=-\delta\bar{\mathfrak{l}}_{k}(X_R)+\frac{\delta(1-\delta)}{2}\Gamma(X,C_{k})^2-\Phi_{\delta}(X,D_{k})$ with $\delta\in(0,1)$, where $\Gamma(X,C)$ and $\Phi_{\delta}(X,D)$ are defined in~\eqref{Eq:Gamma} and~\eqref{Eq:Theta} respectively.

\begin{lemma}
Suppose that there exists $X_R\in\mathcal{B}_{+}(\H_R)$ and $\mathbf{j}\in[m]$ such that $\H_S$ is invariant for $\mathbf{j}$-th subsystem and $\mathcal{E}_{\mathbf{j},\delta}(X)>0$ for some $\delta\in(0,1)$,
Consider the function $V(\rho)=\mathrm{Tr}(X\rho)^{\delta}$ where $X$ is the extension in $\mathcal{B}(\H)$ of $X_R$. Under these conditions, the following statements hold:
\begin{enumerate}
    \item There exist $c_1=c_1(X,\mathbf{d}_0)>0$ and $c_2=c_2(X,\mathbf{d}_0)>0$, such that for all $\rho\in\mathcal{S}(\H)$,
$
\big(\mathbf{d}_0(\rho)/c_2\big)^{2\delta}\leq V(\rho) \leq \big(\mathbf{d}_0(\rho)/c_1\big)^{\delta}.
$
\item There exists $\bar{r}=\bar{r}(X,C_{\mathbf{j}},D_{\mathbf{j}},\delta)>0$, such that for all $l\in (0,\bar{r})$, a positive constant $c(l)\in(0,\mathcal{E}_{k,\delta}(X))$ can be found, satisfying $\mathcal{A}_{\mathbf{j}}V(\rho)\leq -c(l) V(\rho)$ for $\rho\in \Lambda_{l}$.
\end{enumerate}
Consequently, $\mathbf{\overline{H}1}$ is satisfied. 
\label{Lemma:GES_structure}
\end{lemma}
\begin{proof}
The first statement can be easily verified by using Lemma~\ref{Lemma:Relation_DisLya}. Next, we show the second statement.
The infinitesimal generator of $V(\rho)$ relative to $\mathbf{j}$-th subsystem is given by
\begin{align*}
\mathcal{A}_{\mathbf{j}}V(\rho)=V(\rho)\Bigg[&\delta\frac{\Tr(X\mathcal{F}_{\mathbf{j}}(\rho))}{\Tr(X\rho)}-\frac{\delta(1-\delta)}{2}\left|\frac{\Tr(X\mathcal{G}_{\mathbf{j}}(\rho))}{\Tr(X\rho)}\right|^2\\
&+\left(\left(\frac{\Tr(X \mathcal{J}_{D_\mathbf{j}}(\rho))}{\Tr(X\rho)}\right)^{\delta}-(1-\delta)-\delta\frac{\Tr(X \mathcal{J}_{D_\mathbf{j}}(\rho))}{\Tr(X\rho)}\right)v_{D_{\mathbf{j}}}(\rho) \Bigg].
\end{align*}
From the definition of $\bar{\mathfrak{l}}_{\mathbf{j}}(X_R)$, Lemma~\ref{Lemma:PositiveLim_G} and Lemma~\ref{Lemma:PositiveLim_Q}, we have the following estimates  
$$
\Tr(X\mathcal{F}_{\mathbf{j}}(\rho))=\Tr(\mathcal{F}^*_{\mathbf{j},R}(X_R)\rho_R)\leq \bar{\mathfrak{l}}_{\mathbf{j}}(X_R)\Tr(X_R\rho_R)= \bar{\mathfrak{l}}_{\mathbf{j}}(X_R)\Tr(X\rho)
$$
and 
\begin{align*}
    &\limsup_{\rho\rightarrow \mathcal{I}(\H_S)}\left[\left(\frac{\Tr(X \mathcal{J}_{D_\mathbf{j}}(\rho))}{\Tr(X\rho)}\right)^{\delta}-(1-\delta)-\delta\frac{\Tr(X \mathcal{J}_{D_\mathbf{j}}(\rho))}{\Tr(X\rho)}\right]v_{D_{\mathbf{j}}}(\rho)-\frac{\delta(1-\delta)}{2}\left|\frac{\Tr(X\mathcal{G}_{\mathbf{j}}(\rho))}{\Tr(X\rho)}\right|^2\\
    &\leq \Phi_{\delta}(X,D_{\mathbf{j}})-\frac{\delta(1-\delta)}{2}\Gamma(X,C_{\mathbf{j}})^2.
\end{align*}
It implies that
$
\limsup_{\rho\rightarrow\mathcal{I}(\H_S)}\mathcal{A}_{\mathbf{j}}V(\rho)/V(\rho)\leq -\mathcal{E}_{\mathbf{j},\delta}(X)<0.
$
Then, we can conclude the proof by applying the continuity of $V(\rho)$.    
\end{proof}

Next, we consider a special case where $\mathbf{j}$-th subsystem is undergoing Quantum Non-Demolition (QND) measurements~\cite{benoist2014large,liang2024model}. Consider a decomposition of the whole Hilbert space: $\H=\H_S\oplus\H^1_{R}\oplus \dots \oplus \H^{\mathbf{d}}_{R}$ with $\mathbf{d}\leq d-1$. Denote by $\Pi_0,\Pi_1,\dots,\Pi_{\mathbf{d}}$ the orthogonal projections on $\H_S,\H^1_{R}, \dots,\H^{\mathbf{d}}_{R}$ respectively. We then make the following assumption.
\begin{itemize}
    \item[$\mathbf{\overline{H}4}$:]  $H_0$, $H_{\mathbf{j}}$, $L_{\mathbf{j}}$, $C_{\mathbf{j}}=\sum^{\mathbf{d}}_{i=0}c_{i}\Pi_{i} $ and $D_{\mathbf{j}}=\sum^{\mathbf{d}}_{i=0}a_{i}\Pi_{i}$ with $c_i,a_i\in \mathbb{C}$ are simultaneously block-diagonalizable with respect to the decomposition above.
\end{itemize}

\begin{lemma}
 Assume that $\mathbf{\overline{H}4}$ is satisfied. Suppose that at least one of the following two conditions holds:
    \begin{enumerate}
        \item $\underline{c}:=\min_{i\in[\mathbf{d}]}(\mathbf{Re}\{c_i\}-\mathbf{Re}\{c_0\})^2>0$,
        \item $\underline{a}:=\min_{i\in[\mathbf{d}]}(|a_i|-|a_0|)^2>0$.
    \end{enumerate}
Consider the function $V(\rho)=\sum^{\mathbf{d}}_{i=1}\sqrt{\Tr(\Pi_{i}\rho)}$. The following statements hold:
    \begin{enumerate}
    \item There exist $c_1=c_1(\mathbf{d}_0)>0$ and $c_2=c_2(\mathbf{d}_0)>0$, such that for all $\rho\in\mathcal{S}(\H)$,
$
\mathbf{d}_0(\rho)/c_2\leq V(\rho) \leq \sqrt{\mathbf{d}_0(\rho)/c_1}.
$
\item There exists $\bar{r}=\bar{r}(C_{\mathbf{j}},A_{\mathbf{j}})>0$, such that for all $l\in (0,\bar{r})$, a positive constant $c(l)\in(0,(\underline{c}+\underline{a})/2)$ can be found, satisfying $\mathcal{A}_{\mathbf{j}}V(\rho)\leq -c(l) V(\rho)$ for $\rho\in \Lambda_{l}$.
\end{enumerate}
Consequently, $\mathbf{\overline{H}1}$ is satisfied. 
\end{lemma}
\begin{proof}
 The first statement can be easily verified by using Lemma~\ref{Lemma:Relation_DisLya}. Next, we show the second statement. 
The infinitesimal generator of $V(\rho)$ relative to $\mathbf{j}$-th subsystem satisfies
\begin{align*}
\mathcal{A}_{\mathbf{j}}V(\rho)&=-\frac{1}{2}\sum^{\mathbf{d}}_{i=1}\sqrt{\Tr(\Pi_{i}\rho)}\Big[\big(\mathbf{Re}\{c_i\}-\Tr(C_{\mathbf{j}}\rho-\rho C^*_{\mathbf{j}})/2 \big)^2+ \big(|a_i-\sqrt{\Tr(D_{\mathbf{j}}\rho D_{\mathbf{j}}^*)}| \big)^2 \Big]\\
&\leq -\frac{1}{2}\mathbf{C}(\rho)V(\rho),
\end{align*}
where $\mathbf{C}(\rho):=\min_{i\in[\mathbf{d}]}\big\{\big(\mathbf{Re}\{c_i\}-\Tr(C_{\mathbf{j}}\rho-\rho C^*_{\mathbf{j}})/2 \big)^2+ \big(|a_i|-\sqrt{\Tr(D_{\mathbf{j}}\rho D_{\mathbf{j}}^*)} \big)^2\big\}$. Since 
\begin{align*}
    \lim_{\rho \rightarrow \mathcal{I}(\H_S)}\mathbf{C}(\rho) = \min_{i\in[\mathbf{d}]}\big\{\big(\mathbf{Re}\{c_i\}-\mathbf{Re}\{c_0\} \big)^2+ \big(|a_i|-|a_0| \big)^2\big\}\geq \underline{c}+\underline{a}>0,
\end{align*}
we have
$
\limsup_{\rho\rightarrow\mathcal{I}(\H_S)}\mathcal{A}_{\mathbf{j}}V(\rho)/V(\rho)\leq -(\underline{c}+\underline{a})/2<0.
$
Then, we can conclude the proof by applying the continuity of $V(\rho)$.
\end{proof}



By applying the similar arguments as in Proposition~\ref{Proposition:GES},  we provided the sufficient conditions ensuring the almost sure of the target subspace for the switched SME~\eqref{eq:sme_switch}.

\begin{proposition}
Consider the switched SME~\eqref{eq:sme_switch} under $\sigma_2$. Suppose that the assumptions of Lemma~\ref{Lemma:GES_structure} and Lemma~\ref{Lemma:NeverReach_SME} are satisfied, and $\mathbf{\overline{H}2}$ with $l$ determined in Lemma~\ref{Lemma:GES_structure}, Assumption~\ref{ass:partition3}. 
Then, $\mathcal{H}_S$ is GES almost surely with the sample Lyapunov exponent less than or equal to $\big(2\delta\bar{\mathfrak{l}}_{\mathbf{j}}(X)-\delta\Gamma(X,C_{\mathbf{j}})^2+2(\Phi_{\delta}(X,D_{\mathbf{j}})+\Psi_{\delta}(X,D_{\mathbf{j}})\big)/2\delta$, where 
$$
0\geq \Psi_{\delta}(X,D_{\mathbf{j}})\geq \limsup_{\rho\rightarrow \mathcal{I}(\mathcal{H}_S)}\left[\log\left(\left(\frac{\Tr(X\mathcal{J}_{\mathbf{j}}(\rho))}{\Tr(X\rho)}\right)^{\delta}\right)+1-\left(\frac{\Tr(X\mathcal{J}_{\mathbf{j}}(\rho))}{\Tr(X\rho)}\right)^{\delta}\right]\Tr(D_{\mathbf{j}}\rho D_{\mathbf{j}}^*).
$$
\label{Proposition:GES_SME}
\end{proposition}
\begin{proof}
  Lemma~\ref{Lemma:NeverReach_SME} ensures that, for $\mathbf{j}$-subsystem, the related trajectory $\rho_t$ cannot attain the target subspace $\mathcal{I}(\H_S)$ in finite time almost surely. Due to the compactness of $\mathcal{S}(\H)$, the associated condition~\eqref{Eq:Cond_h} for switched SME~\eqref{eq:sme_switch} holds. Thus, by applying the similar arguments as in the proof of Proposition~\ref{Proposition:GES} on the function $V(\rho)=\Tr(X\rho)^{\delta}$ with $\delta\in(0,1)$, we can show the almost sure GES of the target subspace and provide an estimate of the Lyapunov exponent. Now, we investigate the last term in the estimation of the Lyapunov exponent, $\Psi_{\delta}(X,D_{\mathbf{j}})$.
Based on the condition (iv) of Proposition~\ref{Proposition:GES}, for some constant $c>0$ large enough, we have
\begin{align*}
    &\int_{|z|\leq c}\log\left(\frac{V\big(\rho+\mathcal{Q}_{\mathbf{j}}(\rho,z)\big)}{V(\rho)}\right)+1-\frac{V\big(\rho+\mathcal{Q}_{\mathbf{j}}(\rho,z)\big)}{V(\rho)} \d z\\
    &=\left[\log\left(\left(\frac{\Tr(X\mathcal{J}_{\mathbf{j}}(\rho))}{\Tr(X\rho)}\right)^{\delta}\right)+1-\left(\frac{\Tr(X\mathcal{J}_{\mathbf{j}}(\rho))}{\Tr(X\rho)}\right)^{\delta}\right]\Tr(D_{\mathbf{j}}\rho D_{\mathbf{j}}^*)\leq 0.
\end{align*}
Based on the invariance assumption of $\mathbf{j}$-th subsystem and proof of Proposition~\ref{Lemma:PositiveLim_Q}, we have
\begin{align*}
    \frac{\Tr(X\mathcal{J}_{\mathbf{j}}(\rho))}{\Tr(X\rho)} = \frac{\Tr(D^*_{\mathbf{j},R}X_RD_{\mathbf{j},R}\rho_R)}{\Tr(X_R\rho_R)\Tr(D_{\mathbf{j}}\rho D^*_{\mathbf{j}})}.
\end{align*}
If $D^*_{\mathbf{j},S}D_{\mathbf{j},S}>0$, then 
\begin{align*}
    \underline{r}(X,D_{\mathbf{j}}) \leq \lim_{\rho\rightarrow \mathcal{I}(\H_S)} \frac{\Tr(D^*_{\mathbf{j},R}X_RD_{\mathbf{j},R}\rho_R)}{\Tr(X_R\rho_R)\Tr(D_{\mathbf{j}}\rho D^*_{\mathbf{j}})}\leq \bar{r}(X,D_{\mathbf{j}})
\end{align*}
where $\underline{r}(X,D_{\mathbf{j}}):=\frac{\underline{\lambda}(D^*_{\mathbf{j},R} X_R D_{\mathbf{j},R})}{\bar{\lambda}(D^*_{\mathbf{j},S}D_{\mathbf{j},S})\bar{\lambda}(X_R)}$ and $\bar{r}(X,D_{\mathbf{j}}):=\frac{\bar{\lambda}(D^*_{\mathbf{j},R} X_R D_{\mathbf{j},R})}{\underline{\lambda}(D^*_{\mathbf{j},S}D_{\mathbf{j},S})\underline{\lambda}(X_R)}$. Define $g_{\delta}(x) = \delta\log x+1-x^{\delta}$, which is non-positive for all $x\geq 0$ and equal to zero if and only if $x=1$. We deduce,
\begin{align*}
    &\limsup_{\rho\rightarrow \mathcal{I}(\H_S)}\left[\log\left(\left(\frac{\Tr(X\mathcal{J}_{\mathbf{j}}(\rho))}{\Tr(X\rho)}\right)^{\delta}\right)+1-\left(\frac{\Tr(X\mathcal{J}_{\mathbf{j}}(\rho))}{\Tr(X\rho)}\right)^{\delta}\right]\Tr(D_{\mathbf{j}}\rho D_{\mathbf{j}}^*)\\
    &\leq 
    \begin{cases}
        g_{\delta}\big( \bar{r}(X,D_{\mathbf{j}}) \big)\underline{\lambda}(D^*_{\mathbf{j},S}D_{\mathbf{j},S})<0,& \text{if }\bar{r}(X,D_{\mathbf{j}})<1;\\
        g_{\delta}\big( \underline{r}(X,D_{\mathbf{j}}) \big)\underline{\lambda}(D^*_{\mathbf{j},S}D_{\mathbf{j},S})<0,& \text{if }\underline{r}(X,D_{\mathbf{j}})>1.
    \end{cases}
\end{align*}  
\end{proof}


\section{Conclusion}

This paper addresses the stabilization of jump-diffusion stochastic systems, both classical and quantum, via a novel hysteresis switching strategy. We rigorously establish the well-posedness of switched jump-diffusion SDEs and SMEs, and derive sufficient conditions to ensure global asymptotic or exponential stability of a designated target state. In contrast to conventional approaches that require global Lyapunov functions, either common or multiple, or impose stability conditions on every subsystem, our method leverages local Lyapunov-like arguments and state-dependent switching to achieve global stabilization with finite switching almost surely.

We further apply this framework to quantum feedback control systems modeled by jump-diffusion stochastic master equations, incorporating both diffusive and jump dynamics. The proposed method relaxes the restrictive invariance and convergence assumptions commonly found in previous work, thus enhancing the practical applicability of quantum control in experimental settings.

Future directions include extending the approach to handle mismatched initial states and model uncertainties~\cite{liang2021robustness}, as well as unmodeled dissipative effects~\cite{liang2025exploring}, which pose significant challenges to traditional feedback designs. Switching control offers a promising path forward in these realistic scenarios. Additionally, we aim to incorporate constraints such as minimum or average dwell time between switching events to avoid excessively frequent mode transitions.

\appendix
\section{Appendix}

\paragraph{Infinitesimal generator}
Given a stochastic differential equation $dx(t)= f(x(t-))dt+g(x(t-))dW(t)+\int_{|z|\leq c}h(x(t-),z)\tilde{N}(dt,dz)$, where $x(t)$ takes values in $Q\subset \mathbb{R}^d,$ the infinitesimal generator is the operator $\mathcal{A}$ acting on twice continuously differentiable functions $V: Q\times \mathbb{R}_{+}  \rightarrow \mathbb{R}$ in the following way
\begin{align}
\mathcal{A}V(x,t):=\frac{\partial V(x,t)}{\partial t}&+\mathfrak{D}_xV f(x)+\frac12 \mathrm{Tr}\Big(g^*(x)  \mathfrak{D}^2_x V  g(x)\Big)\nonumber\\
&+\int_{|z|\leq c}\Big(V\big(x+h(x,z)\big)-V(x)-\mathfrak{D}_xV h(x,z) \Big) \nu(dz),\label{Eq:ItoGenerator}
\end{align}
where 
\begin{equation*}
    \mathfrak{D}_xV=\left[\frac{\partial V}{\partial x_1},\dots, \frac{\partial V}{\partial x_d} \right], \quad \mathfrak{D}^2_xV=\left[ \frac{\partial^2 V}{\partial x_i\partial x_j}\right]_{d\times d}.
\end{equation*}

\paragraph{Stochastic LaSalle invariance theorem}
Here, we briefly review the stochastic LaSalle invariance theorem of Kushner~\cite{kushner1972stochastic,kushner1968invariant}. Let the state space $\mathcal{Q}$ be a complete separable metric space, and $x(t)$ be a homogeneous strong Markov process on $\mathcal{Q}$ starting from $x_0$, whose distribution is continuous in $t$. For $V\in\mathcal{C}(\mathcal{Q},\mathbb{R}_{\geq 0})$ and $\alpha>0$, define $\mathcal{Q}_{\alpha}:=\{x\in\mathcal{Q}|\,V(x)<\alpha\}$ and assume $\mathcal{S}_{\alpha}$ is nonempty. Denote by $\tilde{\mathcal{A}}_{\alpha}$ the weak infinitesimal operator~\cite[Chapter 1.6]{dynkin1965markov} of the stopped process $x(t\wedge \tau_{\alpha})$, where $\tau_{\alpha}:=\inf\{t\geq0|\, x(t)\notin \mathcal{Q}_{\alpha}\}$. Suppose $V$ is in the domain of $\tilde{\mathcal{A}}_{\alpha}$.
\begin{theorem}
Let $\tilde{\mathcal{A}}_{\alpha} V\leq 0$ in $\mathcal{Q}_{\alpha}$. Then, $V(x(t))$ converges for almost all sample paths remaining in $\mathcal{Q}_{\alpha}$. Moreover, if $\mathcal{Q}_{\alpha}$ has compact closure, $x(t\wedge \tau_{\alpha})$ is Feller continuous and for any $\epsilon>0$, $\lim_{t\rightarrow t_0}\mathbb{P}(\|x(t\wedge \tau_{\alpha})-x_0\|\geq \epsilon)=0$ uniformly for $x_0\in\mathcal{Q}_{\alpha}$. Then, $x(t)$ converges in probability to the largest invariant set contained in $\{x\in\mathcal{S}_{\alpha}|\,\tilde{\mathcal{A}}_{\alpha}V(x)=0\}$ for almost all sample paths which never leave $\mathcal{Q}_{\alpha}$. 
\label{Thm:LaSalle}
\end{theorem}

\paragraph{Instrumental lemmas for the main results}
Here, we state instrumental results that is used in the proof of the GAS/GES of the target subspace almost surely for the switched SME~\eqref{eq:sme_switch}. In the following lemma, based on the generator~\eqref{Eq:ItoGenerator}, we provide the infinitesimal generator with respect to SME~\eqref{eq:sme0}.
\begin{lemma}
    For any twice continuously differentiable function $V: \mathcal{S}(\mathbb{H})\times \mathbb{R}_{+}  \rightarrow \mathbb{R}$, the corresponding infinitesimal generator with respect to SME~\eqref{eq:sme0} is given by
    \begin{align}
       \mathcal{A}V(\rho,t)= \frac{\partial V(\rho,t)}{\partial t}&+\Tr\big(\mathfrak{D}_{\rho}V \mathcal{F}(\rho)\big)+\frac{1}{2}\Tr\big(\mathfrak{D}^2_{\rho}V\,\mathcal{G}_{C}(\rho)^2\big)\nonumber\\
       &+\Big( V\big(\mathcal{J}_{D}(\rho)\big)-V(\rho)-\Tr\big(\mathfrak{D}_{\rho}V \mathcal{H}_{D}(\rho)\big)\Big)v_{D}(\rho),\label{Eq:GeneratorSME}
\end{align}
where $\mathfrak{D}_{\rho}V$ denotes the first differential in $\rho$ of $V$ and $\mathfrak{D}^2_{\rho}V$ the second differential, which can be expanded in terms of partial derivatives as in the generator of type~\eqref{Eq:ItoGenerator}.
\end{lemma}

\proof
Based on the generator~\eqref{Eq:ItoGenerator}, for the function $V$, we have
\begin{align*}
       \mathcal{A}V(\rho,t)= \frac{\partial V(\rho,t)}{\partial t}&+\Tr\big(\mathfrak{D}_{\rho}V \mathcal{F}(\rho)\big)+\frac{1}{2}\Tr\big(\mathfrak{D}^2_{\rho}V\,\mathcal{G}_{C}(\rho)^2\big)\\
       &+\int_{|z|<c}V\big(\rho+\mathcal{Q}_{D}(\rho,z)\big)-V(\rho)-\Tr\Big(\mathfrak{D}_{\rho}V \mathcal{Q}_{D}(\rho,z) \Big) dz.
\end{align*}
where $\mathcal{Q}_{D}(\rho,z)=\mathds{1}_{\{0<z<v_{D}(\rho)\}}\mathcal{H}_{D}(\rho,z)$.
It is straightforward to compute that 
\begin{align*}
   & V\big(\rho+\mathcal{Q}_{D}(\rho,z)\big)-V(\rho)-\Tr\Big(\mathfrak{D}_{\rho}V \mathcal{Q}_{D}(\rho,z) \Big) \\
    &= \Big[V\big(\mathcal{J}_{D}(\rho)\big)-V(\rho)-\Tr\Big(\mathfrak{D}_{\rho}V \mathcal{H}_{D}(\rho,z)\big)\Big)\Big] \mathds{1}_{\{0<z<v_{D}(\rho)\}},
\end{align*}
and 
\begin{equation*}
    \int_{|z|<c} \mathds{1}_{\{0<z<v_{D}(\rho)\}} dz = v_{D}(\rho).
\end{equation*}
Then, we can obtain the generator~\eqref{Eq:GeneratorSME}.
\hfill$\square$

The following lemma is analogous to Lemma~\ref{Lemma:NeverReach} for the classical Jump-Diffusion SDE.
\begin{lemma}
Suppose that there exists $\mathbf{j}\in[m]$ such that $\mathbb{H}_S$ is invariant for $\mathcal{F}_{\mathbf{j}}$. In addition, assume $D^*_{\mathbf{j},R}D_{\mathbf{j},R}>0$ when $D_{\mathbf{j}}\neq 0$.
For all $\rho_0\in\mathcal{I}(\mathbb{H}_S)$, 
$
\mathbb{P}\big( \rho(t)\in \mathcal{I}(\mathbb{H}_S),\,\forall t\geq0\big)=1,
$
and for all 
$\rho_0\in\mathcal{S}(\mathbb{H})\setminus\mathcal{I}(\mathbb{H}_S)$, 
$
\mathbb{P}\big( \rho(t)\in \mathcal{S}(\mathbb{H})\setminus\mathcal{I}(\mathbb{H}_S),\,\forall t\geq0\big)=1,
$
where $\rho(t)$ is a solution of the controlled system~\eqref{eq:sme_switch} under switching law $\sigma_2$ starting from $\rho_0\in\mathcal{S}(\mathbb{H})$.
\label{Lemma:NeverReach_SME}
\end{lemma}

\proof
It is obvious that $\mathbb{H}_S$ is invariant in mean for $\mathbf{j}$-th subsystem if $\mathbf{\overline{H}1}$ is satisfied. Thus, $\mathbb{H}_S$ is invariant almost surely for $\mathbf{j}$-th subsystem by employing the similar arguments as in~\cite[Theorem 1.1]{benoist2017exponential} and~\cite[Lemma 3.2]{applebaum2009asymptotic}.

Let us now prove the second part of the lemma. Let
\begin{equation*}
\rho_{R,red}=
\begin{cases}
\frac{\rho_R}{\mathrm{Tr}(\rho_R)},& \text{if }\mathrm{Tr}(\rho_R)\neq 0;\\
\mu_R,& \text{if }\mathrm{Tr}(\rho_R)= 0,
\end{cases}
\end{equation*}
where $\mu_R\in\mathcal{S}(\mathbb{H}_R)$ is arbitrary and $\rho_{R,red}\in\mathcal{S}(\mathbb{H}_R)$. Define $V(\rho)=\frac{1}{\mathrm{Tr}(\Pi_{0}^{\perp}\rho)}$ where $\Pi_{0}^{\perp}=\mathbf{I}-\Pi_0$. 
The infinitesimal generator of $V(\rho)$ with respect to $\mathbf{j}$-th subsystem is given by
\begin{align*}
  \mathcal{A}_{\mathbf{j}}V(\rho) = V(\rho) \Big( -\mathrm{Tr}\big(\mathcal{F}_{\mathbf{j},R}(\rho_{R,red})\big)+\mathcal{T}_{\mathbf{j}}(\rho)^2+\Tr(D_{\mathbf{j}}\rho D_{\mathbf{j}}^*)^2/\mathcal{R}_{\mathbf{j}}(\rho)+\mathcal{R}_{\mathbf{j}}(\rho) \Big)
\end{align*}
$\mathcal{T}_{\mathbf{J}}(\rho):=\mathrm{Tr}\big((C_{\mathbf{J},R}+C_{\mathbf{J},R}^*)\rho_{R,red}\big)- \mathrm{Tr}\big((C_{\mathbf{J}}+C_{\mathbf{J}}^*)\rho \big)$ and $\mathcal{R}_{\mathbf{j}}(\rho):=\Tr(D_{\mathbf{j},R}^*D_{\mathbf{j},R}\rho_R)/\Tr(\rho_R)$.

Due to the compactness of $\mathcal{S}(\mathbb{H})$ and $\mathcal{S}(\mathbb{H}_R)$, for any $\zeta\in(0,l^*-\epsilon)$, where $l^*$ and $\epsilon$ are defined in $\mathbf{\overline{H}2}$, and $\rho\in\Lambda_{l^*-\epsilon}\setminus\Lambda_{\zeta}$, there exists a finite constant $\alpha\geq 0$ such that
$
 -\mathrm{Tr}\big(\mathcal{F}_{\mathbf{j},R}(\rho_{R,red})\big)+\mathcal{T}_{\mathbf{j}}(\rho)^2 \leq \alpha.
$
Moreover, $\Tr(D_{\mathbf{j}}\rho D_{\mathbf{j}}^*)\leq \|D_{\mathbf{j}}^*D_{\mathbf{j}}\|_{HS}$ due to Cauchy-Schwarz inequality, and 
$$
0<\underline{\lambda}(D_{\mathbf{j},R}^*D_{\mathbf{j},R})\leq \mathcal{R}_{\mathbf{j}}(\rho)\leq \bar{\lambda}(D_{\mathbf{j},R}^*D_{\mathbf{j},R})
$$
since $D_{\mathbf{j},R}^*D_{\mathbf{j},R}>0$. Thus, $\Tr(D_{\mathbf{j}}\rho D_{\mathbf{j}}^*)^2/\mathcal{R}_{\mathbf{j}}(\rho)+\mathcal{R}_{\mathbf{j}}(\rho)\leq \beta$ for some finite constant $\beta\geq0$.

Therefore, for all $\rho\in \Lambda_{l^*-\epsilon}\setminus\Lambda_{\zeta}$, there exists a constant $\gamma>0$ such that
$
    \mathcal{A}_{\mathbf{j}}V(\rho)\leq \gamma V(\rho).
$
Furthermore, due to the compactness of $\mathcal{S}(\mathbb{H})$, it is straightforward to show that, for all $k\in[m]$ and $\rho\in\mathcal{S}(\mathbb{H})\setminus \Lambda_{l^*-\epsilon}$,  $\mathcal{A}_{k}V(\rho)\leq \gamma V(\rho)$ for some constant $\gamma>0$.

Define $\tau_{\zeta}:=\inf\{t\geq 0| \rho(t)\in \Lambda_{\zeta}\}$. By applying It\^o formula to $e^{\gamma t}V(\rho(t))$, we have
\begin{equation*}
\mathbb{E}\big( e^{-\gamma (\tau_{\zeta}\wedge t)}V(\rho(\tau_{\zeta}\wedge t))\big)-V(\rho_0)= \mathbb{E} \int^{\tau_{\zeta}\wedge t}_0 \sum^{m}_{k=1}u^{k}_s e^{-\gamma s}\big(\mathcal{A}_k V(\rho(s-))-\gamma V(\rho(s-))\big)ds \leq  0.
\end{equation*}
Conditioning to the event $\{\tau_{\zeta}\leq t\}$, and $V(\rho(\tau_{\zeta}\wedge t))=V(\rho(\tau_{\zeta}))=1/\zeta$, which implies
\begin{equation*}
\frac{1}{\zeta e^{\gamma t}}\mathbb{P}(\tau_{\zeta}\leq t)\leq \mathbb{E} \big(e^{-\gamma (\tau_{\zeta}\wedge t)} V(\rho(\tau_{\zeta}\wedge t))\mathds{1}_{\{\tau_{\zeta}\leq t\}}\big)\leq V(\rho_0).
\end{equation*} 
Thus, $\mathbb{P}(\tau_{\zeta}\leq t) \leq V(\rho_0)\zeta e^{\gamma t}$. Let $\zeta$ converge to zero, we have $\mathbb{P}(\tau_{\zeta}\leq t)$ converge to zero. The proof is complete.
\hfill$\square$

\smallskip

The following result is instrumental to prove the results concerning on the almost sure GES of the target subspace $\mathbb{H}_S$. Before stating the result, we define
\begin{equation}\label{Eq:Gamma}
\Gamma(X,C):=
\begin{cases}
\underline{\lambda}(C_S+C_S^*)-\max\left\{0, \frac{\bar{\lambda}(Z_R)}{\underline{\lambda}(X_R)}\right\},&\text{if }\max\left\{0, \frac{\bar{\lambda}(Z_R)}{\underline{\lambda}(X_R)}\right\}<\underline{\lambda}(C_S+C_S^*)\\
\min\left\{0, \frac{\underline{\lambda}(Z_R)}{\underline{\lambda}(X_R)}\right\}-\bar{\lambda}(C_S+C_S^*),&\text{if }\min\left\{0, \frac{\underline{\lambda}(Z_R)}{\underline{\lambda}(X_R)}\right\}>\bar{\lambda}(C_S+C_S^*)\\
0,&\text{else},
\end{cases}
\end{equation}
and 
\begin{equation*}
\Gamma(C):=
\begin{cases}
\underline{\lambda}(C_S+C_S^*),&\text{if }(C_S+C_S^*)>0\\
\bar{\lambda}(C_S+C_S^*),&\text{if }(C_S+C_S^*)<0\\
0,&\text{else},
\end{cases}
\end{equation*}
where $Z_R=X_RC_R+C_R^*X_R$ is Hermitian, and define
\begin{equation}\label{Eq:Theta}
    \Phi(X,D):=\left(\frac{\bar{\lambda}(D^*_R X_R D_R)}{\underline{\lambda}(X_R)}\right)^\delta\bar{\lambda}(D^*_S D_S)^{1-\delta}-(1-\delta)\underline{\lambda}(D^*_S D_S)-\delta\frac{\underline{\lambda}(D^*_R X_R D_R)}{\bar{\lambda}(X_R)}.
\end{equation}

\begin{lemma}
For the noise operators $C\in\mathcal{B}(\mathbb{H})$ such that $C_Q=0$ and $X_R\in\mathcal{B}_{+}(\mathbb{H}_R)$, if one of the following is satisfied
\begin{enumerate}
\item[\emph{(i)}] $\max\{0, \bar{\lambda}(Z_R)/\underline{\lambda}(X_R)\}<\underline{\lambda}(C_S+C_S^*)$;
\item[\emph{(ii)}] $\min\{0, \underline{\lambda}(Z_R)/\underline{\lambda}(X_R)\}>\bar{\lambda}(C_S+C_S^*)$,
\end{enumerate} 
then 
$
\liminf_{\rho\rightarrow\mathcal{I}(\mathbb{H}_S)}|\Tr(X\mathcal{G}_C(\rho))/\Tr(X\rho)|^2\geq \Gamma(X,C)^2>0,
$ 
where $X$ is the extension in $\mathcal{B}(\mathcal{H})$ of $X_R$. In particular, for the case $X=\Pi_{0}^{\perp}$ where $\Pi_{0}^{\perp}$ is the projection on $\mathbb{H}_R$, if $(C_S+C_S^*)>0$ or $(C_S+C_S^*)<0$, then 
$
\liminf_{\rho\rightarrow\mathcal{I}(\mathbb{H}_S)}|\Tr(\Pi_{0}^{\perp}\mathcal{G}_C(\rho))/\Tr(\Pi_{0}^{\perp}\rho)|^2\geq \Gamma(C)^2>0.
$ 
\label{Lemma:PositiveLim_G}
\end{lemma}
\proof
By a straightforward calculation under the assumption $C_Q=0$, we have
\begin{equation*}
\Tr(X\mathcal{G}_C(\rho))=\Tr(Z_R\rho_R)-\Tr((C+C^*)\rho)\Tr(X\rho).
\end{equation*}
First, we investigate the condition (i). If $Z_R\leq 0$, then $\Tr(X\mathcal{G}_C(\rho))\leq -\Tr((C+C^*)\rho)\Tr(X\rho)$. If $\bar{\lambda}(Z_R)>0$ and $Z_R\leq \big(\bar{\lambda}(Z_R)/\underline{\lambda}(X_R)\big) X_R$, then
$$
\Tr(X\mathcal{G}_C(\rho))\leq \Big(\frac{\bar{\lambda}(Z_R)}{\underline{\lambda}(X_R)}-\Tr((C+C^*)\rho)\Big)\Tr(X\rho).
$$
Based on the block-decomposition, we have
\begin{equation*}
\Tr((C+C^*)\rho)=\Tr((C_S+C^*_S)\rho_S)+\Tr((C_R+C^*_R)\rho_R)+\Tr(C_P\rho_P^*+\rho_PC^*_P).
\end{equation*}
Due to $\rho_S\geq 0$, $\Tr((C_S+C^*_S)\rho_S)\geq \underline{\lambda}(C_S+C^*_S)\Tr(\rho_S)\geq 0$ when $(C_S+C^*_S)>0$. Moreover, since $\Tr(\rho_S)$ converges to one, $\rho_P$ and $\rho_R$ converge to zero when $\rho$ tends to $\mathcal{I}(\mathbb{H}_S)$, the condition (i) guarantees that, 
\begin{equation*}
\begin{split}
&\lim_{\rho\rightarrow\mathcal{I}(\mathbb{H}_S)}\Tr((C+C^*)\rho)-\max\left\{0, \frac{\bar{\lambda}(Z_R)}{\underline{\lambda}(X_R)}\right\}\\
=&\lim_{\rho\rightarrow\mathcal{I}(\mathbb{H}_S)}\Tr((C_S+C^*_S)\rho_S)-\max\left\{0, \frac{\bar{\lambda}(Z_R)}{\underline{\lambda}(X_R)}\right\}\geq \underline{\lambda}(C_S+C_S^*)-\max\left\{0, \frac{\bar{\lambda}(Z_R)}{\underline{\lambda}(X_R)}\right\}>0.
\end{split}
\end{equation*}
Thus,
$
\lim_{\rho\rightarrow\mathcal{I}(\mathbb{H}_S)}|\Tr(X\mathcal{G}_C(\rho))/\Tr(X\rho)|^2\geq \Gamma(X,C)^2>0,
$ 
if the condition (i) holds.

Next, we investigate the condition (ii). If $Z_R\geq 0$, then $\Tr(X\mathcal{G}_C(\rho))\geq -\Tr((C+C^*)\rho)\Tr(X\rho)$. If $\underline{\lambda}(Z_R)<0$ and $Z_R\geq \big(\underline{\lambda}(Z_R)/\underline{\lambda}(X_R)\big) X_R$, then
$$
\Tr(X\mathcal{G}_C(\rho))\geq \Big(\frac{\underline{\lambda}(Z_R)}{\underline{\lambda}(X_R)}-\Tr((C+C^*)\rho)\Big)\Tr(X\rho).
$$
By similar arguments, the desired result can be concluded under the condition (ii). 

Now, we suppose $X=\Pi_{0}^{\perp}$. Based on the block-decomposition, one deduces that 
\begin{equation*}
\begin{split}
\Tr(\Pi_{0}^{\perp}\mathcal{G}_C(\rho))&= \big(\Tr((C_R+C^*_R)\rho_R)-\Tr((C+C^*)\rho)\big)\Tr(\Pi_{0}^{\perp}\rho)\\
&=\big(\Tr((C_S+C^*_S)\rho_S)+\Tr(C_P\rho_P^*+\rho_PC^*_P)\big)\Tr(\Pi_{0}^{\perp}\rho).
\end{split}
\end{equation*}
Since $\rho_S\geq 0$ and $\Tr(\rho_S)$ converges to one, $\rho_P$ converge to zero when $\rho$ tends to $\mathcal{I}(\mathbb{H}_S)$,
if $(C_S+C_S^*)>0$ or $(C_S+C_S^*)<0$, then
$
\lim_{\rho\rightarrow\mathcal{I}(\mathbb{H}_S)}|\Tr((C_S+C^*_S)\rho_S)|\geq |\Gamma(C)|>0.
$ 
Then the proof is complete.
\hfill$\square$

\begin{lemma}
Consider the noise operators $D\in\mathcal{B}(\mathbb{H})$ such that $D_Q=0$ and $X_R\in\mathcal{B}_{+}(\mathbb{H}_R)$. For any $\delta\in(0,1)$,
then 
$$
\lim_{\rho\rightarrow\mathcal{I}(\mathbb{H}_S)}\left[\left(\frac{\Tr(X \mathcal{J}_{D}(\rho))}{\Tr(X\rho)}\right)^{\delta}-(1-\delta)-\delta\frac{\Tr(X \mathcal{J}_{D}(\rho))}{\Tr(X\rho)}\right]v_D(\rho)\leq \Phi_{\delta}(X,D),
$$
where $X$ is the extension in $\mathcal{B}(\mathbb{H})$ of $X_R$.
\label{Lemma:PositiveLim_Q}
\end{lemma}

\proof
By a straightforward calculation under the assumption $D_Q=0$, we have
\begin{equation*}
\frac{\Tr(X\mathcal{J}_{D}(\rho))}{\Tr(X\rho)} = \frac{\Tr(D^*_{R}X_RD_R\rho_R)}{v_D(\rho)\Tr(X_R\rho_R)}.
\end{equation*}
It implies 
\begin{align*}
    &\left[\left(\frac{\Tr(X \mathcal{J}_{D}(\rho))}{\Tr(X\rho)}\right)^{\delta}-(1-\delta)-\delta\frac{\Tr(X \mathcal{J}_{D}(\rho))}{\Tr(X\rho)}\right]v_D(\rho)\\
    &=\left(\frac{\Tr(D^*_{R}X_RD_R\rho_R)}{\Tr(X_R\rho_R)}\right)^{\delta}v_D(\rho)^{1-\delta}-(1-\delta)v_D(\rho)-\delta\frac{\Tr(D^*_{R}X_RD_R\rho_R)}{\Tr(X_R\rho_R)}.
\end{align*}
Given $X_R>0$, it follows $D^*{R}X_RD_R\geq 0$, which implies 
$$
\frac{\underline{\lambda}(D^*_R X_R D_R)}{\bar{\lambda}(X_R)}\leq\frac{\Tr(D^*_{R}X_RD_R\rho_R)}{\Tr(X_R\rho_R)}\leq \frac{\bar{\lambda}(D^*_R X_R D_R)}{\underline{\lambda}(X_R)}.
$$
Moreover, we have $\lim_{\rho\rightarrow\mathcal{I}(\mathcal{H}_S)}\Tr(D\rho D^*)=\Tr(D^*_SD_S\rho_S )$ where $\Tr(\rho_S)=1$ for all $\rho\in \mathcal{I}(\mathbb{H}_S)$.
Combining these elements, we conclude the proof of the lemma.
\hfill$\square$

\bibliographystyle{plain}
\bibliography{Ref_SwitchingFeedback}
\end{document}